\documentclass[a4paper]{amsart}
\usepackage{amssymb}
\usepackage[initials]{amsrefs}
\usepackage{cpambib}

\usepackage{a4wide,enumerate,graphicx,todonotes}
\usepackage[normalem]{ulem} 
\usepackage{mathrsfs}
\usepackage{tabulary}
\usepackage{hyperref}
\usepackage[all]{xy}

\hypersetup{
    colorlinks,
    linkcolor={red!50!black},
    citecolor={blue!50!black},
    urlcolor={blue!80!black}
}

\newcounter{abb}
\newcommand{\Abb}[1]{\refstepcounter{abb} \centering \footnotesize \emph{Fig.~\arabic{abb}:} #1}

\newcounter{tab}
\newcommand{\Tab}[1]{\refstepcounter{tab} \centering \footnotesize \emph{Tab.~\arabic{tab}:} #1}

\newcommand{\RR}{\mathscr{R}}
\newcommand{\QQ}{\mathscr{Q}}

\newcommand{\UU}{\mathscr U}
\newcommand{\WW}{\mathscr W}
\newcommand{\R}{\mathbb{R}}
\newcommand{\N}{\mathbb{N}}
\newcommand{\NN}{\mathscr{N}}
\newcommand{\eps}{\varepsilon}
\newcommand{\I}{\mathscr{I}}

\newcommand{\Ft}{\tilde F}

\newcommand{\ft}{\tilde f} 
\newcommand{\myicon}{$\,\,\,\triangleright$}
\newcommand{\diam}{\mathrm{diam}}
\newcommand{\supp}{\mathrm{supp}} 
\newcommand{\loc}{\mathrm{loc}}

\newcommand{\F}{\mathbf{F}}
 
\newcommand{\nX}{\hat\nabla}
\newcommand{\dist}{\mathrm{dist}} 
\newcommand{\vol}{\mathrm{vol}}
 
\newcommand{\im}{\mathrm{im}\,} 
\DeclareMathOperator{\grad}{\mathrm{grad}}
\DeclareMathOperator{\id}{\mathrm{id}}

\newtheorem{thm}{Theorem}[section]
\newtheorem{satz}[thm]{Theorem}
\newtheorem{lem}[thm]{Lemma}
\newtheorem{prop}[thm]{Proposition}
\newtheorem{addendum}[thm]{Addendum} 
\newtheorem{cor}[thm]{Corollary} 
\newtheorem{setting}{Setting}

\theoremstyle{definition}
\newtheorem{rem}[thm]{Remark}
\newtheorem{epi}[thm]{Epilogue} 
\newtheorem{dfn}[thm]{Definition} 
\newtheorem{exa}[thm]{Example} 
\newtheorem{notation}[thm]{Notation}

\begin{document}

\title{Local Flexibility for Open Partial Differential Relations} 
\author{Christian B\"ar}
\address{Universit\"at Potsdam, Institut f\"ur Mathematik, 14476 Potsdam, Germany}
\email{\href{mailto:cbaer@uni-potsdam.de}{cbaer@uni-potsdam.de}}
\urladdr{\url{https://www.math.uni-potsdam.de/baer/}}
\author{Bernhard Hanke}
\address{Universit\"at Augsburg, Institut f\"ur Mathematik, 86135 Augsburg, Germany}
\email{\href{mailto:hanke@math.uni-augsburg.de}{hanke@math.uni-augsburg.de}}
\urladdr{\url{https://www.math.uni-augsburg.de/prof/diff/arbeitsgruppe/hanke/}}
\begin{abstract}
We show that local deformations, near closed subsets, of solutions to open partial differential relations  can be extended to global deformations, provided all but the highest derivatives stay constant along the subset.
The applicability of this general result is illustrated  by a number of examples, dealing with convex embeddings of hypersurfaces, differential forms, and lapse functions in Lorentzian geometry.

The main application is a general approximation result by sections which have very restrictive local properties on open dense subsets.
This shows, for instance, that given any $K \in \R$ every manifold of dimension at least two carries a complete $C^{1,1}$-metric which, on a dense open subset, is smooth with constant sectional curvature $K$. 
Of course, this is impossible for $C^2$-metrics in general. 
\end{abstract}



\maketitle

\section{Introduction}
In his landmark monograph~\cite{Gromov86} Gromov develops a wide-ranging perspective on flexibility phenomena in geometry and topology. 
Besides a comprehensive theoretical background and numerous results, including the  well-known $h$-principle for open Diff-invariant relations and the convex integration technique, Gromov's text provides many exercises which seemingly play a minor role for the architecture of the theory, but some of which bear a great value of their own, both in terms of theoretical insight and applications. 
This article is devoted to one such topic, the \textit{Local Flexibility Lemma}, see the exercise ``Weak Flexibility Lemma'' in \cite[Section~ 2.2.7~(H')]{Gromov86} and the ``Cut-off Homotopy Lemma'' in \cite[p.~693 f.]{Gromov2018}.
It concerns extensions of local deformations of solutions to open partial differential relations. 

We will formulate and prove this flexibility lemma in Theorem~\ref{thm:FlexLem1}.
It has important applications in many fields of mathematics, which will be illustrated by examples from hypersurface theory, geometric structures induced by differential forms, and Lorentzian geometry.

Our main application of Theorem~\ref{thm:FlexLem1} is Theorem~\ref{thm:solveondense} which states under relatively mild assumptions that any section of a fiber bundle can be approximated by sections which have very restrictive local properties on open dense subsets.

We give three sample applications of Theorem~\ref{thm:solveondense}.
Firstly, we show in Corollary~\ref{cor:Lipschitz} that $C^1$-functions on a compact interval can be uniformly approximated by Lipschitz functions which are smooth on open dense subsets and have prescribed derivative there.

Secondly, we show in Corollary~\ref{cor:einbett} that any $C^2$-embedding of a surface in $\R^3$ can be $C^1$-approximated by $C^{1,1}$-embeddings which are analytic and have prescribed constant Gauss curvature on an open dense subset.
Obviously, $C^{1,1}$ cannot be replaced by $C^2$ in this statement. 
In other words: $C^{1,1}$ is the maximal order of regularity  for which this kind of flexibility holds. 
This is reminiscent of the Nash-Kuiper embedding theorem (\cite{Kuiper,Nash}) which states that each short smooth embedding of a compact Riemannian $n$-manifold $V$ into $\R^k$ with $k\ge n+1$ can be $C^0$-approximated by \emph{isometric} $C^1$-embeddings. 
Here the critical exponent $\alpha$ for which approximating isometric $C^{1, \alpha}$-embeddings exist is unknown and subject  to current research, see \cite{CLS2012} and subsequent work.

Thirdly, we show in Corollary~\ref{satz:approxpos} that given $K \in \R$ any $C^2$-Riemannian metric on a manifold $V$ of dimension at least two can be approximated in the strong $C^1$-topology by $C^{1,1}_\loc$-metrics which, on open dense subsets of $V$, are smooth with constant sectional curvature equal to $K$.  
Clearly, this approximation cannot be done by $C^2$-metrics because then the curvature would be continuous and equal to $K$ on all of $V$, which is only possible in exceptional cases. 

This means, for example, that a compact surface of higher genus carries a $C^{1,1}$-metric which, on an open dense subset,  is smooth with constant  Gauss curvature equal to $1$, despite the fact that the Gauss-Bonnet theorem holds for these metrics.
Indeed, this  is not a contradiction because open dense subsets need not have full measure.
However, it is remarkable that the relevant curvature information entering the Gauss-Bonnet formula can  be concentrated on a nowhere dense subset, although this information governs the global topology. 

To formulate local flexibility precisely we will work in the following

\begin{setting}
\label{setting}
We denote by
\begin{enumerate}[\myicon]
\item 
$V$ a smooth manifold;
\item
$V_0 \subset V$ a closed subset;
\item
$U$ an open neighborhood of $V_0$ in $V$;
\item
$X\to V$ a smooth fiber bundle;
\item
$k\in\N_0$ a nonnegative integer;
\item
$\RR\subset J^kX$ an open subset;
\item
$f_0$ a $C^{k}$-section on $V$, solving $\RR$;
\item
$F :[0,1] \to C^k(U,X)$ a continuous path such that each $F(t)$ solves $\RR$ over $U$.
\end{enumerate}
Furthermore, we assume $f_0|_U = F(0)$ and $j^{k-1}F(t)|_{V_0}  = j^{k-1}f_0|_{V_0}$ for all $t \in [0,1]$.
\end{setting}

Here $J^kX\to V$ denotes the $k^\mathrm{th}$ jet bundle of $X$ and $j^kf$ is the $k$-jet of a section $f$.
We say that a $C^k$-section $f$ of $X\to V$ \emph{solves} $\RR$ if $j^kf(v)\in\RR$ for all $v\in V$.
By $C^k(U,X)$ we denote the space of $k$-times continuously differentiable sections of the bundle $X|_U\to U$ equipped with the weak $C^k$-topology.
We use the notation $\N=\{1,2,3,\ldots\}$ and $\N_0=\{0,1,2,3,\ldots\}$.

\begin{center}
\includegraphics[scale=0.3]{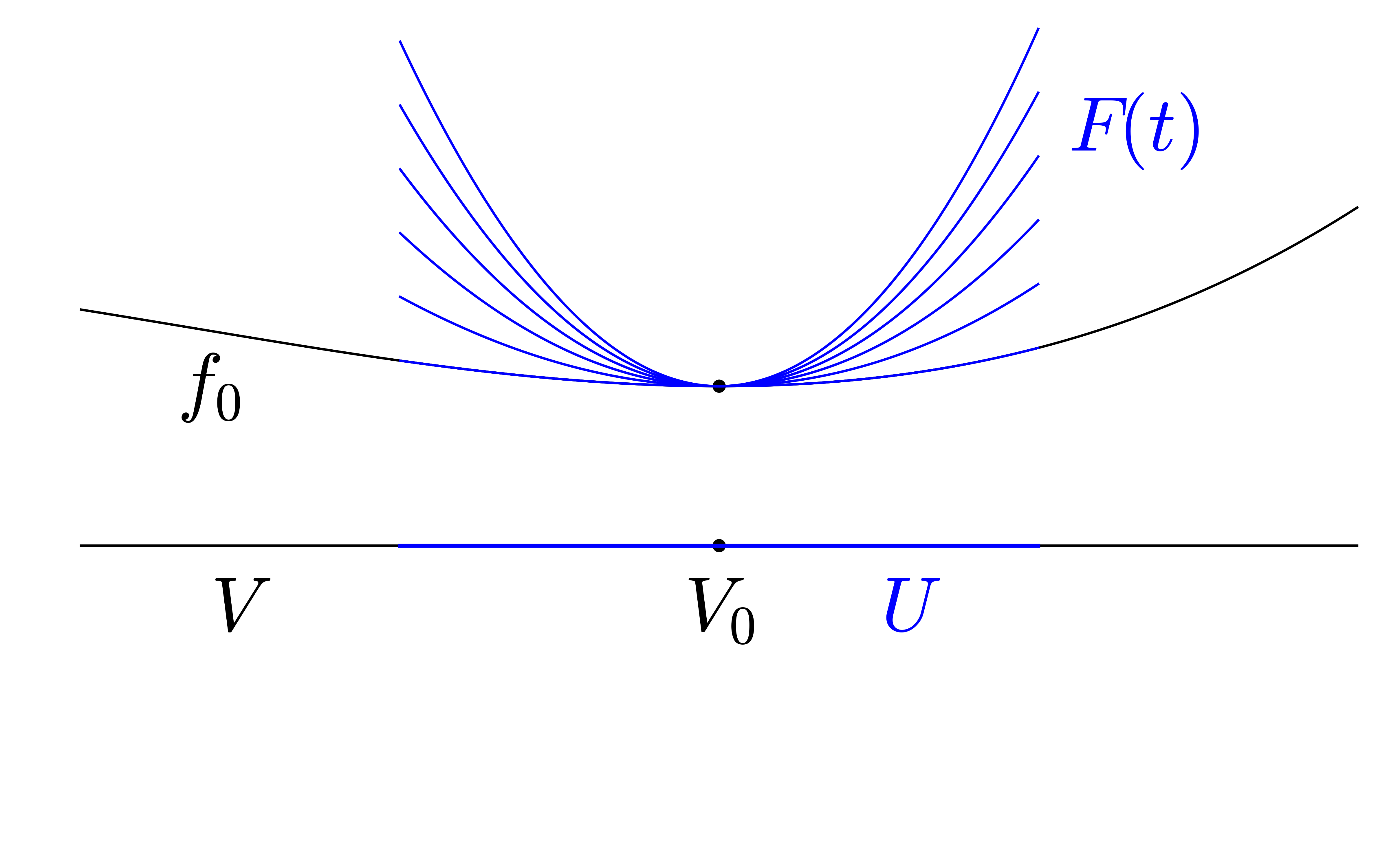}

\vspace{-5mm}
\Abb{Setting~\ref{setting} for $k=2$ and $V_0$ a point}
\label{fig:1}
\end{center}

\begin{dfn}\label{def:GromovFlexible}
In Setting~\ref{setting} we say that \emph{local flexibility holds} if there exists an open subset $U_0$ with $V_0\subset U_0 \subset U\subset V$ and a continuous $f:[0,1]\to C^k(V,X)$ such that 
\begin{enumerate}[\myicon]
\item 
each $f(t)$ is a section of $X$, solving $\RR$;
\item
$f(0)=f_0$;
\item
$f(t)|_{U_0}=F(t)|_{U_0}$ for all $t\in[0,1]$;
\item
$f(t)|_{V\setminus U}=f_0|_{V\setminus U}$ for all $t\in[0,1]$.
\end{enumerate}
\end{dfn}

\begin{center}
\includegraphics[scale=0.3]{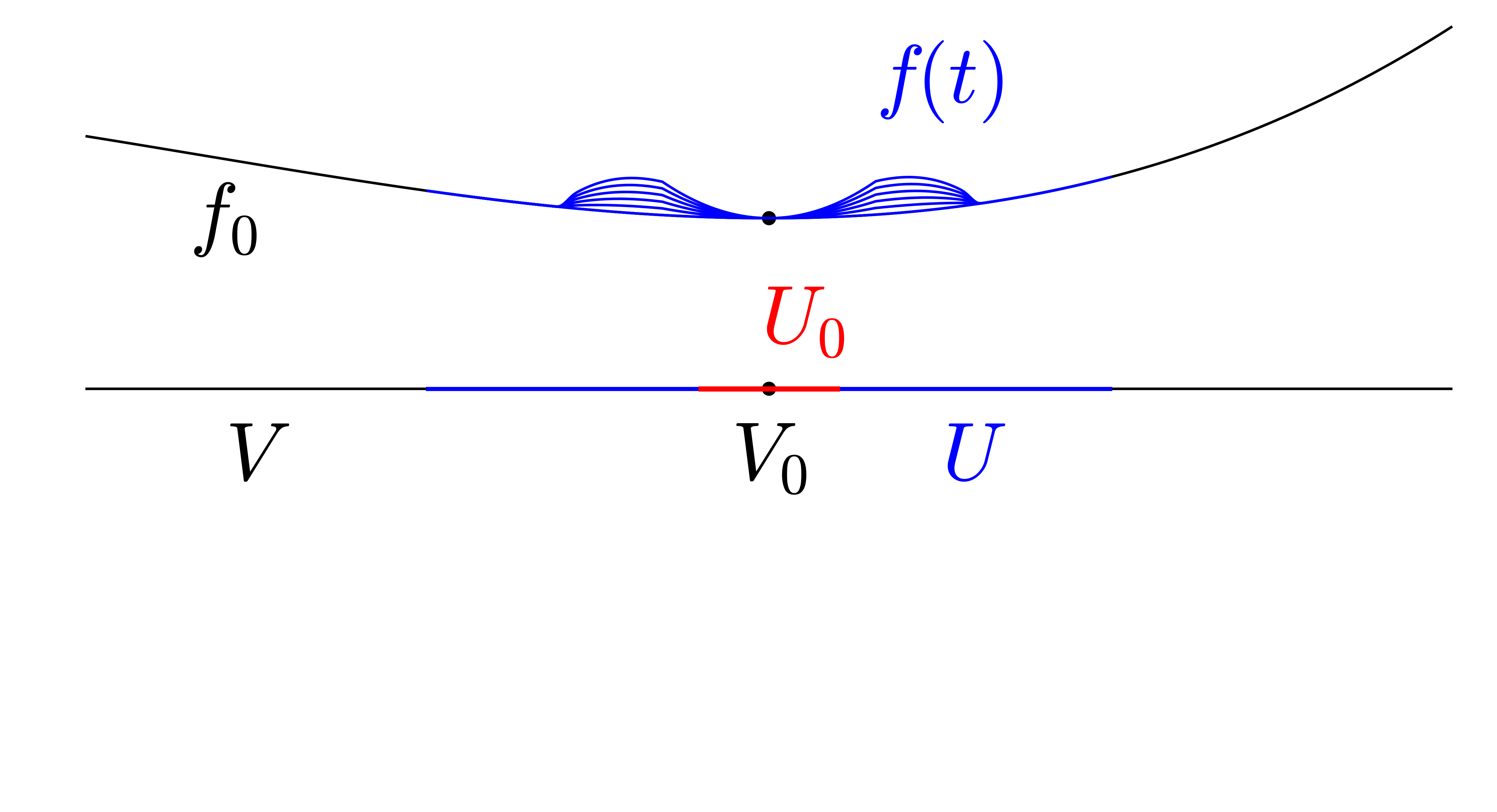}

\vspace{-8mm}
\Abb{Local flexibility holds for $k=2$ and $V_0$ a point}
\label{fig:2}
\end{center}

\begin{thm}\label{thm:FlexLem1}
Suppose we are in Setting~\ref{setting}. 
Then local flexibiliy holds. 

Moreover, let $\kappa\in\{k,k+1,\ldots,\infty\}$, $\ell\in\{0,1,\ldots,\infty\}$, and assume in addition that $f_0\in C^\kappa(V,X)$ and $F\in C^\ell([0,1],C^\kappa(U,X))$. 
Then we can find $f\in C^\ell([0,1],C^\kappa(V,X))$.
\end{thm}

\begin{rem} \label{rem:Cheeky_Misha} The ``Weak Flexibility Lemma'' in \cite[Section~ 2.2.7~(H')]{Gromov86} and the ``Cut-off Homotopy Lemma'' in  \cite[p.~693 f.]{Gromov2018} formulate (without proof) versions of Theorem~\ref{thm:FlexLem1} under strong regularity assumptions on $V_0$. 
The footnote on \cite[p. 693]{Gromov2018} speculates about local flexiblity for all closed subsets $V_0 \subset V$.
This is what is proved in the present paper.

While  the applications worked out in the paper at hand  are based on local flexibility for smooth submanifolds $V_0 \subset V$, we expect that local flexibility for more general closed subsets $V_0 \subset V$ may also have interesting  consequences. 

\end{rem} 

\begin{rem} \label{rem:Other_CutOff} The sections $f(t)$ are obtained from $F(t)$ by multiplying the homotopy parameter $t$ with an appropriate cutoff function near $V_0$. 
It is then relatively straighforward to control the $(k-1)$-jets of $f(t)$ over $V$; compare \cite[Theorem 1.5]{AG18} for a related result in the context of holonomic approximations near polyhedral subsets $V_0 \subset V$ of positive codimension. 

For us  it is crucial that a careful choice of cutoff function allows to control the full $k$-jets of $f(t)$ over $V$, using the assumption that $j^{k-1} F(t)|_{V_0}$ is constant in $t$. 
If $V_0 \subset V$ is a compact smooth submanifold we may in fact use a cutoff function $\tau_{\delta, \eps}(r)$ as in Lemma  \ref{lem:DefFunction}, where $r$ is the distance to $V_0$. 
\end{rem} 

Our paper is structured as follows:
The next section introduces the notion of generalized tangent spaces along arbitrary subsets of smooth manifolds and applies this concept to construct efficient cutoff functions near these subsets. 
This construction is essential for the proof of local flexibility in the subsequent section, in which we also discuss the necessity of the assumptions in Setting~\ref{setting}, treat a family version of local flexibility  and provide a homotopy theoretic interpretation.

In the fourth section we illustrate the usefulness of local flexibility by examples from different mathematical contexts.
We start by considering the standard sphere $S^n\subset \R^{n+1}$.
It is $1$-convex and rigid in the sense that one cannot deform it in such a way that it becomes $\mu$-convex near the north pole for some $\mu>1$ while keeping it unchanged on the southern hemisphere and $1$-convex everywhere.
However, local flexibility shows that a deformation is possible if we only demand that it stays $(1-\eps)$-convex everywhere.

In the second example we deform closed differential forms satisfying an open relation along a submanifold through such forms.
This applies in particular to symplectic forms where we recover a statement usually derived using the so-called Moser trick.
The method also applies to closed $G_2$-structures on $7$-manifolds, where the result is new, and to codimension-$1$-foliations in any dimension.

Then we deal with Lorentzian manifolds.
We show that given a spacelike Cauchy hypersurface $\Sigma$ we can find a Cauchy time function in such a way that we can prescribe the lapse function along~$\Sigma$.

The fifth section is devoted to formulate and prove Theorem~\ref{thm:solveondense} together with the applications to Lipschitz functions and surface embeddings. 
The proof of Theorem~\ref{thm:solveondense} is by repeated application of local flexibility to $V_0$ being a point from a countable dense subset of $M$ and passage to a limit.

Our main application of Theorem~\ref{thm:solveondense} to Riemannian metrics is Corollary~\ref{satz:approxpos} in the sixth section concerning the existence of $C^{1,1}$-metrics which, on open dense subsets, are smooth and of constant sectional curvature $K \in \R$. 
In this section we also use the family version of local flexibility to show that, on a fixed manifold $V$ and point $p\in V$, the inclusion of metrics with positive sectional curvature which is equal to $K>0$ near $p$ into the space of all positively curved metrics is a weak homotopy equivalence. 
The appendices contain the proofs of local flexibility for $k = 0$, of an auxiliary lemma needed for the proof of local flexibility for $k \geq 1$, and of the Gauss-Bonnet theorem for compact surfaces with a metric of low regularity.

\medskip

\textit{Acknowledgment.} 
We are grateful to Misha Gromov for pointing out the relevance of the local flexibility lemma, to Luis Florit and Sebastian Goette for useful conversations, to Burkhard Wilking whose insightful remarks helped us to improve Theorem~\ref{thm:solveondense}, and to an anonymous referee for the suggestions on how to improve the presentation.
The authors were supported by SPP 2026 funded by Deutsche Forschungsgemeinschaft.

\section{Generalized tangent spaces and efficient cutoff functions}

This section contains preparatory material for the proof of Theorem~\ref{thm:FlexLem1}.
We will have to construct efficient cutoff functions near arbitrary closed subsets of a manifold.
Since closed subsets can be very irregular, we introduce the concept of generalized tangent spaces, which mimick classical tangent spaces of submanifolds.
The decay of our cutoff functions will have to be chosen differently in the direction of the tangent spaces and in those perpendicular to them.

\begin{dfn}  \label{dfn:tangent} 
Let $V$ be an $n$-dimensional $C^1$-manifold and let $A \subset V$ be an arbitrary subset. 
For  $a \in A$ let $\mathscr{G}_{A,a}$ denote the set of $C^1$-germs at $a$ which vanish on $A$. 
More precisely, each element in  $\mathscr{G}_{A,a}$ is represented by a $C^1$-function $h : U \to \R$ defined on an open neighborhood $a \in  U \subset V$ such that  $h|_{A \cap U} = 0$. 
Then we call 
\[
      T_a A := \bigcap_{ h \in \mathscr{G}_{A,a}}  \ker( d_a h) \subset T_a V   
\]
the {\em generalized tangent space} of $A$ at $a$. 
Here $d_a h: T_a V \to \R$ is the differential of $h$ at $a$. 
\end{dfn} 

\begin{rem} 
The generalized tangent space $T_aA$ is a linear subspace of $T_aV$.
If $A \subset V$ is a $C^1$-submanifold then this reproduces the classical tangent space of $A$ at $a$.

If $\dim(T_aA)=m$ then we can find $h_1,\ldots,h_{n-m}:U\to\R$ such that $d_ah_1,\ldots,d_ah_{n-m}$ are linearly independent, $T_aA=\ker(d_ah_1)\cap\cdots\cap\ker(d_ah_{n-m})$ and $A$ is contained in the $m$-dimensional $C^1$-submanifold $\{h_1=\cdots=h_{n-m}=0\}$ near $a$.
 
The function $a \mapsto \dim T_a A$ is upper semicontiuous, i.e.\ every $a \in A$ has a neighborhood $a \in U \subset V$ such that $\dim T_a A \geq \dim T_{a'} A$ for all $a' \in A \cap U$. 
\end{rem} 

Let $A \subset V$ and set $\Sigma_{\ell} := \{ a \in  A \mid \dim T_a A \geq n - \ell \}$ for $-1 \leq \ell \leq n$. 
By the upper semicontinuity of $a \mapsto \dim T_a A$, the sets $\Sigma_\ell$ are closed in $A$.
They form a chain 
\begin{equation} \label{eq:chain} 
     \emptyset = \Sigma_{-1} \subset \cdots \subset \Sigma_n = A 
\end{equation} 
with  $\overline{\Sigma_{\ell} \setminus \Sigma_{\ell-1}} \subset \overline{\Sigma}_\ell = \Sigma_{\ell}$ for $0 \leq \ell\leq n$ and where the closures are taken in $A$. 

\begin{dfn}  
A subset $K \subset A$ is called {\em uniform} if $\dim T_a A = \dim T_{a'} A$ for all $a, a' \in K$. 
In other words, $K \subset \Sigma_{\ell}\setminus \Sigma_{\ell-1}$ for some $0 \leq \ell \leq n$.
\end{dfn}

\begin{exa} Let $V = \R^2$. 
\begin{enumerate}[\myicon]
\item 
Let $A=\{(t,|t|) \mid t\in\R\}$.
Then $T_aA=T_aV$ for $a=(0,0)$ while $T_aA$ is the usual $1$-dimensional tangent space for all other points $a\in A$.
In \eqref{eq:chain} we have $\Sigma_0  = \{ ( 0,0) \}$ and $\Sigma_1 =  \Sigma_2 = A$. 
\item 
Let $A = \{ | v | \leq 1\}$. 
Then $T_a A = T_a V$ for all $a \in A$. 
Hence $\Sigma_0=\Sigma_1=\Sigma_2=A$.
\item 
Let $A_1 = \{ ( t, t^2) \mid t \in \R\}$ and $A_2 = \{ (t,0) \mid t \in \R \}$. 
Then, for $a = (0,0)$, we get $T_a A_1 = T_a A_2 = \R \times 0 \subset T_a V$ while $T_a ( A_1 \cap A_2) = 0$ and $T_a (A_1 \cup A_2) = T_a V$. 
The last equation follows from the fact that, near $a$, the set $A_1 \cup A_2$ is not contained in a $1$-dimensional submanifold and hence $\dim(T_aA)>1$.
\item 
Let $A = \{ (1/n, 0) \mid n \in \N\} \cup \{ (0,0)\} $. 
Then $T_a A = 0$ for $a = (1/n,0)$, while $T_a  A = \R \times 0 \subset T_a V$ for $a = (0,0)$. 
In general, for a discrete subset $D \subset V$ and $A := \overline{D}$, we have $T_a A = 0$ for $a \in D$, while $T_a A$ depends on the accumulation behaviour of $D$ near $a$ for $a \in A \setminus D$.
\end{enumerate} 
\end{exa}

In the remainder of this section we will specialize to the case $V = \R^n$. 
Let $A \subset V$ be an arbitrary subset. 
We use the canonical identifications $T_x \R^n = \R^n$ for $x \in \R^n$.

\begin{notation} For $B \subset \R^n$, $a \in A$, and $x \in \R^n$ we write 
\begin{enumerate}[\myicon]
\item $\dist(x, B)$ for the Euclidean distance of $x$ to $B$; 
\item $r_a (x) = \dist( x , a + T_a A)$ for the distance of $x$ to the affine subspace $a + T_a A \subset \R^n$, the generalized tangent space with footpoint $a$; 
\item $B(\eps,x) \subset \R^n$ for the open $\eps$-ball around $x$. 
\end{enumerate}
\end{notation}

\begin{lem} \label{lem:approx} 
Let $K \subset A$ be a compact and uniform subset.
Let $\delta > 0$. 

Then there exists $\eta>0$ such that for all $\eps\in(0,\eta)$ and all $a\in K$ we have
\begin{enumerate}[(i)] 
\item \label{approx-i}
$A\cap B(\eps, a) \subset \{r_{a} < \delta \eps \}$; 
\item \label{approx-ii}
$ |r_{a'} - r_{a} |_{B(\eps, a)} < \delta \eps$ for all $a' \in K \cap B(\eps, a)$.
\end{enumerate}
\end{lem} 

The proof of this statement is simpler if  $A \subset \R^n$ is  a $C^1$-submanifold,  since then $A$ is, locally around $a \in K$, the graph of a $(T_a A)^{\perp}$-valued $C^1$-function over $a + T_a A$ by the implicit function theorem. 

\begin{center}
\includegraphics[scale=0.7]{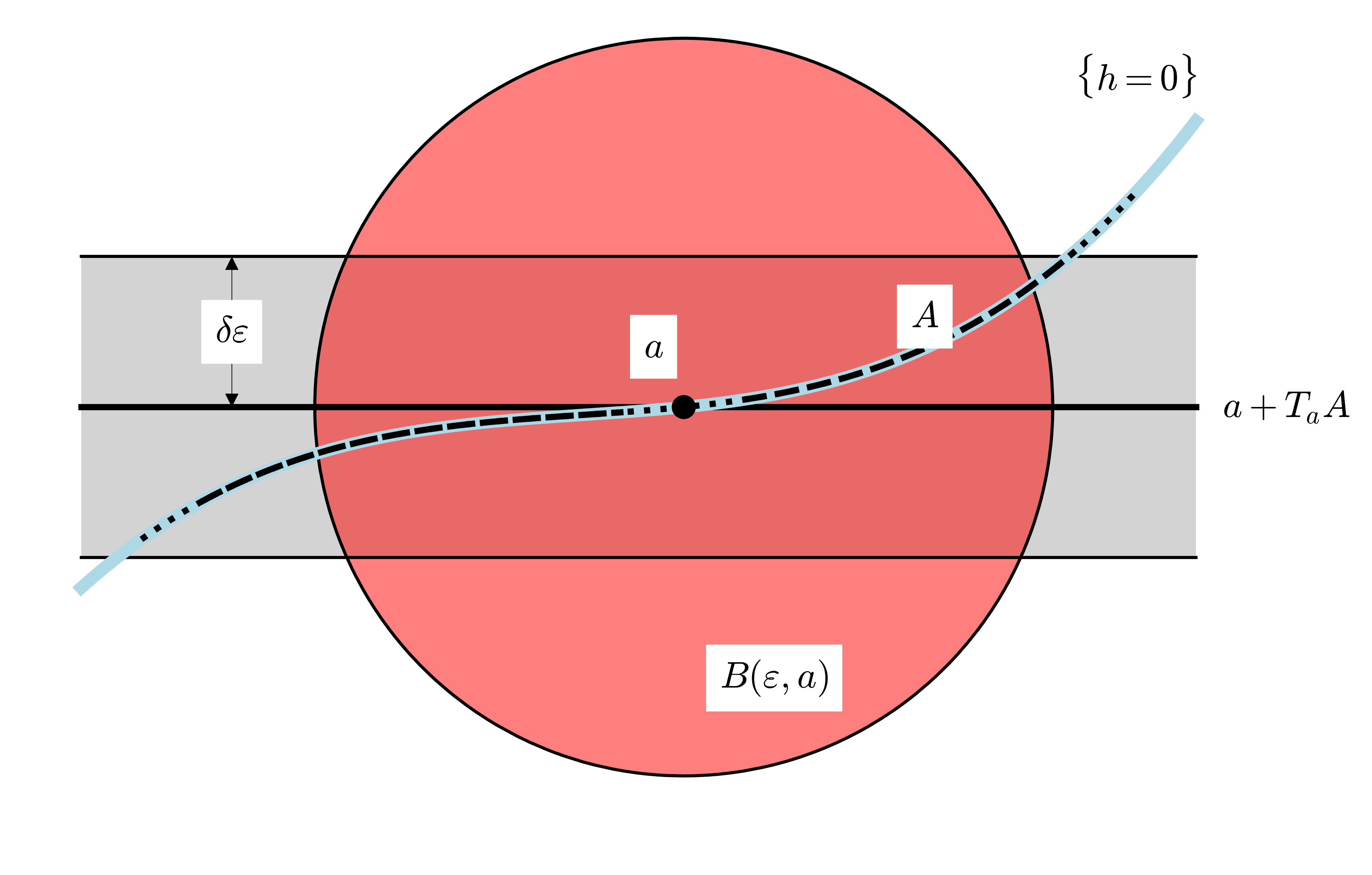}

\Abb{Inclusion $A\cap B(\eps, a) \subset \{r_{a} < \delta \eps \}$}\label{fig:delepsAbsch}
\end{center}

\begin{proof}[Proof of Lemma \ref{lem:approx}] 
Let $a\in K$.
Let  $q := n - \dim T_a A$. 
We say that a $C^1$-map $h:U\to\R^q$ is {\em adapted} to $A$ at $a$ if 
\begin{enumerate}[\myicon]
\item $h$ vanishes on $A \cap U$;
\item $T_aA = \ker(d_ah)$;
\item $d_xh$ has maximal rank for every $x\in U$.
\end{enumerate}
Note that then $T_xA = \ker(d_xh) \subset \R^n$ for all $x \in K \cap U$ because $K$ is a uniform subset of $A$. 
Thus $h$ is adapted to $A$ at each $x\in K\cap U$.

For an adapted $h=(h_1,\ldots,h_q)$ denote by $H_1,\ldots,H_q:U\to \R^n$ the vector fields obtained by the Gram-Schmidt procedure applied to the gradient fields $\nabla h_1,\ldots,\nabla h_q$.
They form a continuous orthonormal frame for the normal bundle of the submanifold $\{h=0\}$.

\uline{\emph{Step~1:}}
We first show that there exists an $\eta_a>0$ (depending on $a$) and a $C^1$-map $h:B(\eta_a,a)\to\R^q$  adapted to $A$ at $a$ with
\begin{gather}
\big\{ h = 0 \big\}  \cap B( \eps', a')   
\subset 
\big\{ y\in B( \eps', a') \mid \dist \big( y , a' + T_{a'} A )\big) <   \delta \eps'/2 \big\} \, ; 
\label{eq:distx} \\
\big|H_j(x)-H_j(y)\big|  < \frac{\delta}{4\cdot\sqrt{q}} \quad\text{ for all }x,y\in B(\eps,a);
\label{eq:dhAbsch}
\end{gather}
for all $0<\eps'\le \eps < \eta_a$ and all $a'\in K\cap B(\eps,a)$.

Indeed, by the definition of generalized tangent space we can find an $h$, adapted to $A$ at $a$, defined on a neighborhood of $\bar B(2\eta_a,a)$.
After possibly decreasing $\eta_a$, the implicit function theorem yields, for every $a'\in B(\eta_a,a)$, a $C^1$-map $g_{a'}: (a'+T_{a'}A)\cap B(\eta_a,a')=(a'+\ker d_{a'}h)\cap B(\eta_a,a')\to \ker (d_{a'}h)^\perp$ such that $\{ h = 0 \}  \cap B( \eta_a, a')$ is contained in the graph of $g_{a'}$.

Let $\tilde U\subset U$ be a compact neighborhood of $a$.
Since $d_{a}h|_{\ker (d_{a}h)^\perp}$ is invertible, the same is true for $d_{y}h|_{\ker (d_{a'}h)^\perp}$ with $y\in\tilde U$ and $a'\in\bar B(\eta_a,a)$ provided $\tilde U$ and $\eta_a$ are sufficiently small.
Decreasing $\eta_a$ further if necessary, we have $x+g_{a'}(x)\in \tilde U$ for $x\in B(2\eta_a,a)\cap (a'+T_{a'}A)$.

Note that $g_{a'}(a')=0$.
Differentiating the equation $h(x+g_{a'}(x))=0$ with respect to $x\in (a'+\ker d_{a'}h)\cap B(\eta_a,a')$ yields 
$$
d_xg_{a'} = -\big(d_{x+g_{a'}(x)}h|_{\ker (d_{a'}h)^\perp}\big)^{-1} \cdot (d_{x+g_{a'}(x)}h|_{\ker (d_{a'}h)}).
$$
Since $(y,a')\mapsto \big(d_{y}h|_{\ker (d_{a'}h)^\perp}\big)^{-1}$ is continuous, it is bounded for $(y,a')\in \tilde U\times \bar B(\eta_a,a)$.
Hence 
$$
\big|d_xg_{a'}\big|\le c\cdot\big|d_{x+g_{a'}(x)}h|_{\ker (d_{a'}h)}\big|  .
$$
For $x=a'$ we have $d_{x+g_{a'}(x)}h|_{\ker (d_{a'}h)}=d_{a'}h|_{\ker (d_{a'}h)}=0$ so that, after decreasing $\eta_a$ once more, we can assume
$$
\big|d_{x+g_{a'}(x)}h|_{\ker (d_{a'}h)}\big| \le \frac{\delta}{3c}
$$
for all $x\in B(\eta_a,a')$ and $a'\in B(\eta_a,a)$.
This implies $\big|d_xg_{a'}\big|\le \frac{\delta}{3}$ and hence $|g_{a'}(x)|\le\frac{\delta}{3}\dist(x,a')$.

Now equation~\eqref{eq:distx} follows:
let $y\in \big\{ h = 0 \big\}  \cap B( \eps', a')$ and write $y=x+g_{a'}(x)$.
Then
$$
\dist( y , a' + T_{a'} A )
\le
\dist(x+g_{a'}(x),x)
=
|g_{a'}(x)|
\le 
\frac{\delta}{3}\cdot\eps'
<
\frac{\delta\eps'}{2} .
$$
Moreover, the vector fields $H_j$ are continuous and hence uniformly continuous on compact sets.
Thus, after possibly decreasing $\eta_a$ one more time, we also get \eqref{eq:dhAbsch}.

\uline{\emph{Step~2:}}
For each $a\in K$ let $\eta_a\in (0,\infty]$ be maximal amongst all constants for which \eqref{eq:distx} and \eqref{eq:dhAbsch} hold.
Let $\lambda\in (0,\eta_a)$ and let $a_1\in K\cap B(\lambda,a)$.
If $0 < \eps'\le\eps<\eta_a-\lambda$ then $B(\eps,a_1)\subset B(\eps+\lambda,a)$ and thus \eqref{eq:distx} holds for all $a'\in K \cap B(\eps,a_1)$.
Similarly, \eqref{eq:dhAbsch} remains valid for all $x,y\in B(\eps,a_1)$.

This shows that $\eta_{a_1}\ge \eta_a -\lambda$ for all $a_1\in B(\lambda,a)$.
Therefore the function $a\mapsto \eta_a$ is lower semicontinuous and hence attains its minimum $\eta>0$ on the compact set $K$.

We now choose $0< \eps'=\eps<\eta$ and $a=a'$ in \eqref{eq:distx} and get
\begin{equation*} 
A \cap B( \eps, a) 
\subset 
\big\{ h = 0 \big\}  \cap B( \eps , a) 
\subset 
\big\{ r_{a} < \delta\eps/2 \big\} 
\subset 
\big\{ r_{a} < \delta\eps \big\}  ,
\end{equation*}
see Figure~\ref{fig:delepsAbsch}.
This shows \eqref{approx-i}.

\uline{\emph{Step~3:}}
Now let ${a'} \in A\cap  B(\eps, a)\subset \{ h = 0 \} \cap B(\eps, a)$ and let $a_0 \in a + \ker d_a h$ denote the orthogonal projection of ${a'}$.
Then  $a_0 \in B(\eps, a)$, and $|{a'} - a_0| < \delta\eps/2$ by \eqref{eq:distx}. 
The triangle and Cauchy-Schwarz inequalities together with \eqref{eq:dhAbsch} imply for $x \in \R^n$:
\begin{align*} 
\big| &\dist(x, {a'} + \ker d_{a'} h) - \dist(x, a + \ker d_a h) \big|  \\
&= 
\big| \dist(x, {a'} + \ker d_{a'} h) - \dist(x, a_0 + \ker d_a h) \big| \\
&\le
\big| \dist(x, {a'} + \ker d_{a'} h) - \dist(x, a_0 + \ker d_{a'} h) \big| + \big| \dist(x, a_0 + \ker d_{a'} h) - \dist(x, a_0 + \ker d_a h) \big| \\
&< 
\delta\eps/2 + \big| \dist(x, a_0 + \ker d_{a'} h) - \dist(x, a_0 + \ker d_a h) \big| \\
&= 
\big| \dist ( x - a_0, \ker d_{a'} h) - \dist( x - a_0, \ker d_a h) \big| + \delta \eps/2 \\
&= 
\bigg| \Big(\sum_{j=1}^q \langle x - a_0, H_j({a'})\rangle^2\Big)^{1/2} - \Big(\sum_{j=1}^q \langle x - a_0, H_j(a)\rangle^2\Big)^{1/2} \bigg| + \delta \eps/2 \\
& \leq \Big( \sum_{j=1}^q \bigl\langle x - a_0 , H_j({a'}) - H_j(a) \bigr\rangle^2 \Big)^{1/2} + \delta \eps/2 \\  
&\leq 
| x - a_0 | \cdot \Big(\sum_{j = 1}^q | H_j({a'}) - H_j(a) |^2\Big)^{1/2} + \delta \eps/2 \\
&\leq 
|x - a_0| \cdot \delta/4  + \delta \eps/2 \, . 
\end{align*} 
If $x \in B(\eps, a)$ then $|x - a_0| < 2 \eps$ and hence 
\[
    \big| r_{a'}(x) - r_{a}(x) \big| <  2 \eps \cdot \delta/4 + \delta \eps/2 =  \delta \eps \, . 
\]
This proves part \eqref{approx-ii}. 
\end{proof}

The following corollary, which we formulate with constants adapted to our later needs, combines the previous estimate with a covering multiplicity bound.

\begin{cor} \label{cor:covering} 
Let $K \subset A$ be a compact and uniform subset.
Let $\delta > 0$. 

Then there is an $\eta>0$ such that for each $\eps \in (0,\eta)$ there exists a  finite family $(a_i)_{i \in I}$ of points in $K$ with the following properties: 
\begin{enumerate}[(i)] 
  \item  \label{covering-i}
  $A \cap B(\eps, a_i) \subset \{ r_{a_i}   < \delta^2 \eps \}$ for $i \in I$; 
  \item  \label{covering-ii}
  $|r_{a_i}  - r_{a_j} |_{B(2\eps,a_i)} <  \delta^2 \eps $ for  $i, j  \in I$ with $| a_i - a_j |  \leq 4 \eps $; 
  \item 
  $K\subset \bigcup_{i\in I}B(\eps, a_i)$;
  \item   
  The multiplicity of the family  $\big(B(2 \eps, a_i) \big)_{i \in I}$  is bounded by $10^n$, i.e.\ each point in $\R^n$ is contained in at most $10^n$ different balls $B(2 \eps, a_i)$. 
\end{enumerate} 
\end{cor} 

\begin{proof} For each $\eps >0 $ we find a maximal family $(a_i)_{i \in I}$ of points in $K$ such that the balls $B( \eps/2, a_i)\subset \R^n$ are pairwise disjoint. 
Then $\big(B( \eps, a_i)\big)_{i \in I} $ covers $K$. 
Moreover,  the elementary volume comparison $\frac{\vol( B(5 \eps, 0))}{\vol( B(\eps/2, 0))} = 10^n$ implies that for each $i \in I$ the ball $B(4 \eps, a_i)$ contains at most $10^n$ points $a_j$,   since otherwise the balls $B(\eps/2, a_j) \subset B(5 \eps, a_i)$ cannot be pairwise disjoint. 
This implies that the multiplicity of $\big(B( 2 \eps, a_i)\big)_{i \in I}$ is bounded by $10^n$. 

By Lemma~\ref{lem:approx} (applied with $\delta^2/2$ instead of $\delta$), assertions \eqref{covering-i} and \eqref{covering-ii} hold as well for sufficiently small $\eta$. 
\end{proof} 

We will now use generalized tangent spaces in order to construct efficient cutoff functions. 
The next lemma is proved in Appendix~\ref{sec-AnhangB}.

\begin{lem}\label{lem:DefFunction}
For $0 < \delta < \tfrac14$ and $0 < \eps < 1$  there are $C^\infty$-functions $\tau_{\delta, \eps}:\R\to \R$ with the following properties:
\begin{enumerate}[(i)]
\item\label{eq:tauepsi}
$\tau_{\delta,\eps}(r)  =1$ for $r\le \delta \eps$;
\item\label{eq:tauepsii}
$\tau_{\delta , \eps}(r)  =0$ for $r\ge\eps $;
\item
$0\le \tau_{\delta, \eps} \le 1$ everywhere;
\item\label{eq:tauepsiv}
for every $k\in\N$ there is a constant $C_k>0$ such that $\big|\tau_{\delta, \eps}^{(k)}(r)\big| \le C_k \cdot r^{-k} \cdot | \ln \delta|^{-1}$ for all $r > 0$.
\end{enumerate}
\end{lem}

In the following we use standard coordinates $(x^1, \ldots, x^n)$ on $\R^n$ and the multiindex notation 
$$
D^{\alpha}
=
\frac{\partial^{|\alpha|}}{(\partial x^1)^{\alpha_1}\cdots(\partial x^n)^{\alpha_n}} .
$$
Let $X\subset\R^n$ be an affine subspace.
For $x \in \R^n$ and $\lambda>0$ we set 
\[
 \Omega_{\lambda, X}(x) := \max \big\{ \lambda \, , \, \dist(x,X)\big\} \in [\lambda, \infty )  \, . 
\]

\begin{cor} \label{cor:der_distance}
Let $X\subset\R^n$ be an affine subspace and let $r(x)=\dist(x,X)$.
For every multiindex $\alpha$ with $|\alpha|\ge1$ there is a constant $C_\alpha>0$ (independent of $X$) such that
\begin{enumerate}[(i)]
\item\label{der_distance_i}
$ | D^{\alpha} ( \tau_{\delta, \eps} \circ r) | \leq C_{\alpha}  \cdot \Omega_{\delta\eps, X}^{- |\alpha|} \cdot | \ln \delta|^{-1} \, ;$
\item\label{der_distance_ii}
$ | D^{\alpha} (\tau_{\delta, \eps} \circ ( r - \eps'))| \leq C_{\alpha} \cdot (\delta \eps)^{-|\alpha|} \cdot |\ln \delta|^{-1}$;
\end{enumerate}
for all $0 < \delta < \tfrac14$, for all $0 < \eps < 1$, and for all $0\le \eps' \le \eps$.
\end{cor} 

\begin{proof} 
We show \eqref{der_distance_i}.
On $\{r<\delta\eps\}$ the function $\tau_{\delta, \eps} \circ r$ is constant so that the estimate is trivial.
On $\{r\ge\delta\eps\}$ we have $\Omega_{\delta\eps, X} = r$.
The distance function satisfies the well-known estimate
\begin{equation}
     | D^{\alpha} r | \leq C'_{\alpha} \cdot r^{1 - |\alpha| } 
\label{eq:Dara}
\end{equation}
outside of $X$.

Induction on $|\alpha|$ shows that $D^{\alpha}(\tau_{\delta, \eps} \circ r)$ is a linear combination of terms of the form 
$$
(\tau_{\delta, \eps}^{(k)}\circ r)\cdot D^{\beta^{(1)}}r\cdots D^{\beta^{(k)}}r
$$ 
where $k\ge1$, $|\beta^{(j)}|\ge1$ and $|\beta^{(1)}|+\ldots+|\beta^{(k)}|=|\alpha|$.
Lemma~\ref{lem:DefFunction}~\eqref{eq:tauepsiv} together with \eqref{eq:Dara} proves \eqref{der_distance_i}.

To show \eqref{der_distance_ii} we put $\tau_1 := \tau_{\delta, \eps} \circ (r-\eps')$.
For $r \leq \eps' + \delta\eps$ we have $\tau_1 = 1$, hence we may assume
\begin{equation} \label{relevant_domain} 
      r \geq \eps' + \delta \eps \, . 
\end{equation} 
This time $D^{\alpha} \tau_1$ is a linear combination of terms of the form 
\[
     \big(\tau_{\delta, \eps}^{(k)} \circ (r - \eps') \big) \cdot D^{\beta^{(1)}} r \cdots D^{\beta^{(k)}} r
\]
where $k\ge1$, $|\beta^{(j)} | \geq 1$ and $|\beta^{(1)}| + \ldots + |\beta^{(k)}| = |\alpha|$. 
The absolute value of each such term is estimated as follows, using Lemma~\ref{lem:DefFunction} and \eqref{eq:Dara}:
\[
\big|  \big(\tau_{\delta, \eps}^{(k)} \circ (r - \eps') \big) \cdot D^{\beta^{(1)}} r \cdots D^{\beta^{(k)}} r \big| 
\leq 
C_k \cdot | r -\eps'|^{-k} \cdot | \ln \delta|^{-1} \cdot C'_{\beta^{(1)}} r^{1 - |\beta^{(1)}|} \cdots C'_{\beta^{(k)}} r^{1 - |\beta^{(k)}|} \, . 
\]
By assumption~\eqref{relevant_domain} we have $|r - \eps'|^{- k} \leq (\delta \eps)^{-k}$ and $r^{1 - |\beta^{(j)}|} \leq (\delta \eps)^{1 - |\beta^{(j)}|}$ for $1 \leq j \leq k$ because all exponents are nonpositive. 
This concludes the proof.
\end{proof} 

For $A\subset\R^n$ and $a\in A$ we write $\Omega_{\lambda, a} := \Omega_{\lambda, a+T_aA}$.

\begin{lem} \label{lem:cutoff} 
Let $K \subset A$ be a compact and uniform subset.
Let $U \subset \R^n$ be an open neighborhood of $K$. 
Let $k\in\N$. 
Then there exists a constant $C_k>0$ such that the following holds:

For each $0 < \delta < \tfrac14$ there exists an $\eta>0$ such that for each $\eps\in (0,\eta)$ there exists a finite family $(a_i)_{i \in I}$ of points in $K$ and a smooth function $\rho : \R^n  \to [0,1]$ with the following properties: 
\begin{enumerate}[(i)] 
\item \label{cutoff-1}
${\rm supp} ( \rho) \subset\bigcup_{i\in I}B(2\eps, a_i) \subset U$;  
\item  \label{cutoff-2}
$\rho \equiv 1$ on some open neighborhood of $K$;
\item  \label{cutoff-3}
$K\subset \bigcup_{i\in I}B(\eps, a_i)$;
\item \label{cutoff-4}
For $ 0 < | \alpha | \leq k$, $ i \in I$, and $x \in B(2 \eps, a_i)$ we have 
  \[
     |D^{\alpha} \rho (x) | \leq C_k \cdot \Omega_{\delta^2\eps, a_i}(x)^{- |\alpha|} \cdot  |\ln \delta|^{-1} \, . 
  \]
\end{enumerate} 
\end{lem} 

\begin{proof}
For $a \in A$ let 
$r_{a^{\perp}}(x) = \dist( x , a + (T_a A)^{\perp})$  denote the distance of $x$ to the affine subspace
$a + (T_a A)^{\perp}$ of $\R^n$. 
For any $0<\eps<1$ we can consider the following smooth functions $\R^n \to [0,1]$: 
\[
  \tau_1 :=  \tau_{\delta , \eps } \circ ( r_{a^{\perp}} - (1-\delta)\eps)   \, , \qquad \tau_2 := \tau_{\delta ,  \delta \eps} \circ r_{a} \,  , \qquad \tau_a  := \tau_{1}  \cdot \tau_{2} \, . 
\]  
The function $\tau_a$ vanishes outside $B(2 \eps, a)$ and satisfies $\tau_a = 1$ on $\{ r_a \leq \delta^2 \eps \} \cap B(\eps, a)$.
\begin{center}
\includegraphics[scale=.7]{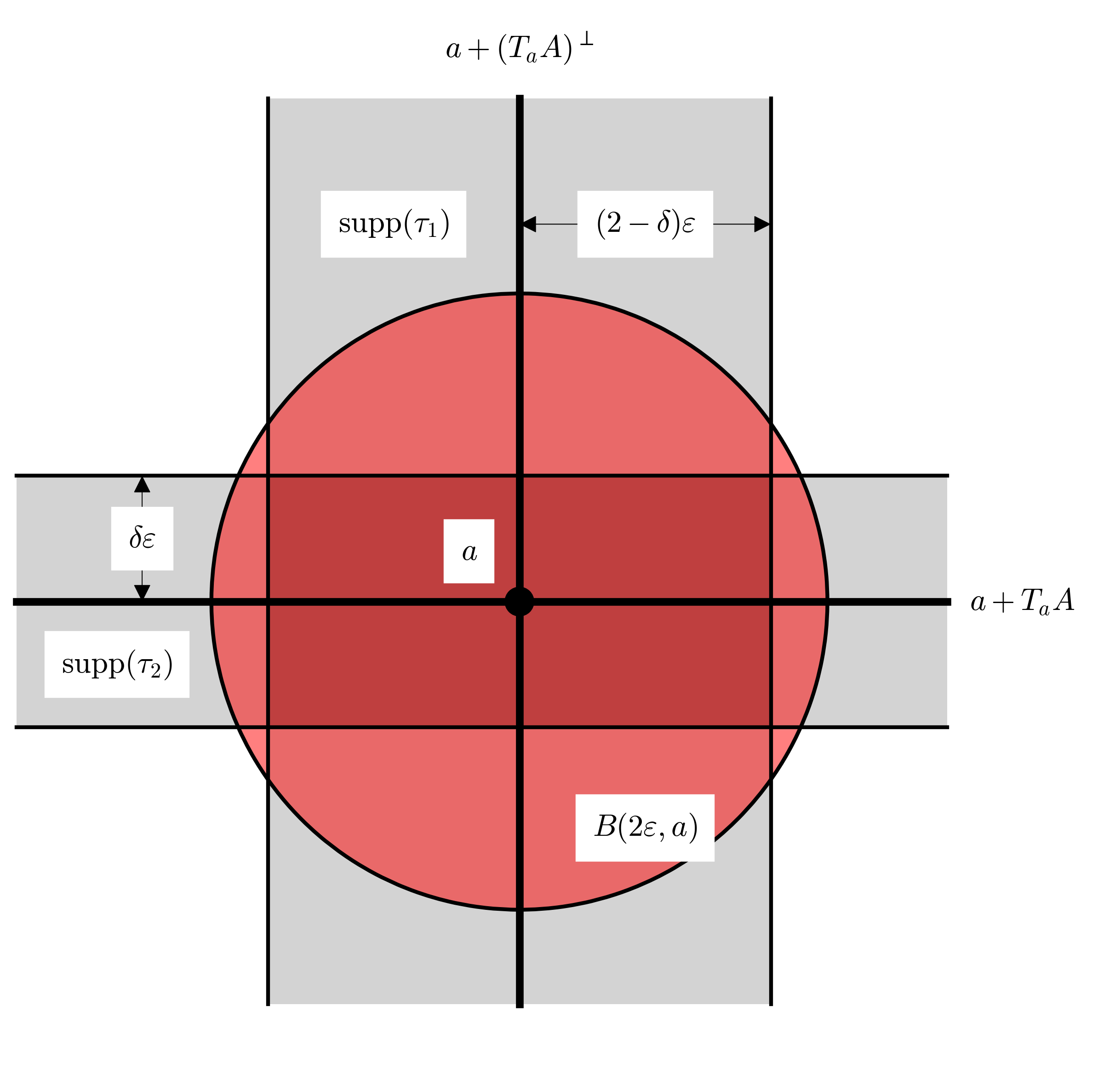}
\includegraphics[scale=.7]{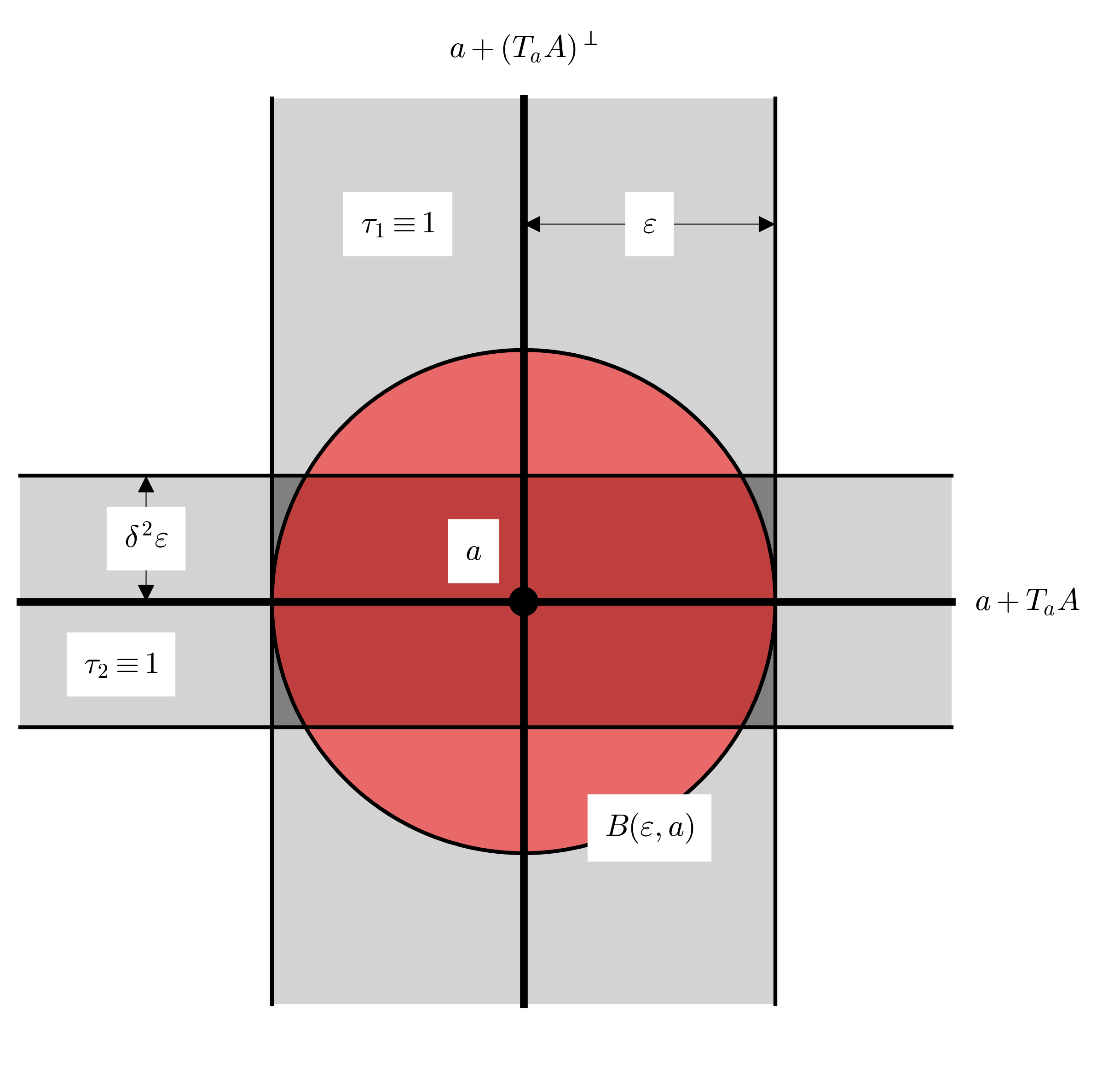}

\Abb{Inclusion $\supp(\tau_a)\subset B(2\eps,a)$}
\hspace{.1\textwidth}
\Abb{Identity $\tau_a \equiv 1$ on $\{ r_a \leq \delta^2 \eps \} \cap B(\eps, a)$}
\vspace{2mm}
\end{center}

Let $0 < |\alpha| \leq k$ and let $x \in \R^n$ with $\tau_a(x) \neq 0$, hence $r_a(x) < \delta \eps$. 
Then, on the one hand, we have 
\[
|  (D^{\alpha} \tau_{1}) (x) |   
\leq  
C_{\alpha}  \cdot   (\delta \eps)^{- | \alpha|}  \cdot |\ln \delta|^{-1} 
\leq 
C_{\alpha} \cdot \Omega_{\delta^2\eps, a}(x)^{- |\alpha|} \cdot | \ln \delta|^{-1} \, , 
\]
where in the first inequality we use Corollary~\ref{cor:der_distance}~\eqref{der_distance_ii} with $X=a + (T_a A)^{\perp}$ and the second inequality uses  $\max \{ \delta^2 \eps, r_a(x) \}  < \delta \eps$.

On the other hand, Corollary~\ref{cor:der_distance}~\eqref{der_distance_i} with $X=a + T_a A$ and $\eps$ replaced by $\delta\eps$ yields
\[
|D^{\alpha} \tau_{2} |    
\leq   
C_{\alpha} \cdot  \Omega_{\delta^2\eps, a}^{- |\alpha|}  \cdot |\ln \delta|^{-1} \, . 
\]
The product rule and $ |\ln \delta|^{-2} \leq |\ln \delta|^{-1}$ for $0 < \delta < \tfrac14$  imply
\begin{equation} \label{eq:esttauv} 
|D^{\alpha} \tau_a| 
\leq 
C_{\alpha}'  \cdot \Omega_{\delta^2\eps, a}^{- |\alpha|} \cdot |\ln \delta|^{-1}  \, . 
\end{equation} 
For the given $\delta$ and for any $\eta$, $\eps$, and $(a_i)_{i \in I}$ as in Corollary~\ref{cor:covering} we now set
\[
       \rho := 1 - \prod_{i \in I} (1 - \tau_{a_{i}})  : \R^n \to \R \, .
\]
It is clear that \eqref{cutoff-1} holds if $\eta$ (and hence $\eps$) is smaller than $\tfrac12\dist(K,\R^n\setminus U)$, which can be assumed without loss of generality. 
Also  \eqref{cutoff-2} is satisfied because the $B(\eps, a_i)$ cover $K$ and for every $i \in I$ we have  $K \cap B(\eps,a_i) \subset \{ r_{a_i} < \delta^2 \eps\}$ by Corollary~\ref{cor:covering}~\eqref{covering-i} and $\tau_{a_i} = 1$ on $\{ r_{a_i} < \delta^2 \eps\} \cap B(\eps, a_i)$. 

Finally, assertion \eqref{cutoff-4} follows from the product rule, estimate \eqref{eq:esttauv} applied to each $\tau_{a_j}$, $j \in I$,  and the following facts: 
\begin{enumerate}[\myicon]
\item 
There are at most $10^n$ indices $j \in J$ with $B(2 \eps, a_i) \cap B(2 \eps, a_j)\neq \emptyset$; 
\item 
If $B(2 \eps, a_i) \cap B(2 \eps, a_j)\neq \emptyset$   we have $|r_{a_i} - r_{a_j} | <  \delta^2 \eps $, hence $\Omega_{\delta^2\eps, a_i} \leq 2  \cdot \Omega_{\delta^2\eps, a_j}$;  
\item 
Higher powers of $|\ln\delta|^{-1}$ can be estimated by $|\ln\delta|^{-1}$.
\qedhere
\end{enumerate}
\end{proof} 

This assertion can be extended to nonuniform subsets as follows:

\begin{lem} \label{lem:cutoffnonpure} 
Let $K \subset A$ be compact but not necessarily uniform.
Let $U$ be an open neighborhood of $K$ in $\R^n$. 
Let $k\in\N$. 

Then there exists a constant $C_k>0$ such that for each $0 < \delta < \tfrac{1}{4}$ and $0 < \Lambda < 1$ there exist finite families $(\eps_i)_{i \in I}$, $0 < \eps_i < \Lambda$, and $(a_i)_{i\in I}$, $a_i \in K$, and a $C^{\infty}$-function $\rho : \R^n \to [0,1]$ with the following properties:
\begin{enumerate}[(i)]  
\item\label{cutoffnonpure-1}
${\rm supp}( \rho) \subset U$;  
\item\label{cutoffnonpure-2}
$\rho = 1$ on some  open neighborhood of $K$;  
\item\label{cutoffnonpure-3}  
For every  $x \in {\rm supp}(\rho)$ and $0 < | \alpha| \leq k$ there is an $i \in I$ with $x \in B(2 \eps_i, a_i)$ and 
\[
|D^{\alpha} \rho (x) | 
\leq 
C_k \cdot  \Omega_{\delta^2\eps_i, a_i}(x)^{- |\alpha|}  \cdot   |\ln \delta|^{-1}  \, .   
\]
\end{enumerate} 
\end{lem} 

\begin{proof} 
We use the chain~\eqref{eq:chain} for an inductive proof. 
For $-1 \leq \ell \leq n$ we set  $K_{\ell} := K \cap \Sigma_{\ell}$, which is a compact subset of $A$. 
We start the induction with $\ell=-1$.
Then $K_{-1}=\emptyset$ and we simply put $\rho_{-1}:=0$.

Assume  $\rho_{\ell-1}$ has been constructed for the compact set $K_{\ell-1} \subset A$ with  index set $I_{\ell-1}$. 
Using the inductive assumption~\eqref{cutoffnonpure-2} we find a compact and uniform subset  $K' \subset K_{\ell} \setminus K_{\ell-1}$ with $K_{\ell}   \subset K' \cup  \{\rho_{\ell-1} = 1\}$.  

We apply Lemma~\ref{lem:cutoff} to $K'$ to obtain $0 < \eps'  < \Lambda$, a family $(a_i)_{i \in I'}$, $a_i \in K'$, and a smooth function $\rho' : \R^n \to [0,1]$ with properties as stated in Lemma~\ref{lem:cutoff} (with $K'$ instead of $K$). 

Set $I_{\ell} := I_{\ell-1} \sqcup I'$, $ \eps_i := \eps'$ for $i \in I'$, and $\rho_{\ell} := 1 - (1- \rho_{\ell-1}) ( 1 - \rho')$. 
Then $\rho$ is a $C^{\infty}$-function $\R^n \to [0,1]$ with properties~\eqref{cutoffnonpure-1} (for small enough $\eps'$) and \eqref{cutoffnonpure-2} (with $K_{\ell}$ instead of $K$). 

For \eqref{cutoffnonpure-3} let $x \in {\rm \supp}(\rho_{\ell})$.
Then there is an $i \in I_{\ell}$ with $x \in B(2 \eps_i, a_i)$.
Amongst all those $i$ choose the one for which $\Omega_{\delta^2\eps_i, a_i}(x)$ attains its minimal value. 
Then \eqref{cutoffnonpure-3} holds by the product rule and by the induction hypothesis for $\rho_{\ell-1}$ (and estimating higher powers of $|\ln \delta|^{-1}$ by $| \ln \delta|^{-1}$). 
\end{proof}  

In the remainder of this section we use generalized tangent spaces to obtain improved Taylor estimates. 
Let $k\geq 1$ and let $F : U \to \R$ be a $C^k$-function defined on a neighborhood $U$ of $A$ in $\R^n$.
Furthermore assume that $j^{k-1} F|_{A\cap U} = 0$. 
For $a \in A\cap U$ let 
\[
\mathscr{T}_{a,k} F (x) 
:=
\sum_{|\beta| \le k} \frac{1}{\beta!} \cdot D^{\beta} F(a)  \cdot (x-a)^{\beta} 
=
\sum_{|\beta| = k} \frac{1}{\beta!} \cdot D^{\beta} F(a)  \cdot (x-a)^{\beta} 
: \R^n  \to \R 
\]
denote the $k^\mathrm{th}$ Taylor polynomial of $F$ at $a$.
The point of the following lemma is the fact that the function $r_a$ in the estimate is not the distance to the point $a$ but the (smaller) distance to the affine space $a+T_aA$.

\begin{lem} \label{lem:taylor} 
Let $K \subset A$ be compact but not necessarily uniform.
Then there is a constant $C_K > 0$ such that for all $a \in K$ and $x \in B(1,a)$ we have 
\[
    |  \mathscr{T}_{a,k} F (x) | \leq C_K  \cdot r_a(x)^{k} \, . 
\]
\end{lem} 

\begin{proof} 
Let $|\beta| = k-1$.
Since $d_a ( D^{\beta} (\mathscr{T}_{a,k} F)) = d_a( D^{\beta} F)$ and $D^{\beta} F : U \to \R$ is a $C^1$-function vanishing on $A \cap U$, the affine map $D^{\beta} ( \mathscr{T}_{a,k} F) : \R^n \to \R  $  vanishes on the affine subspace $a+T_a A \subset \R^n$ by the definition of generalized tangent spaces. 
Using $j^{k-1} \mathscr{T}_{a,k} F (a)  = 0$ this implies, by iterative integration,  that $j^{k-1} \mathscr{T}_{a,k} F |_{a+T_a A} = 0$.

Now let $x \in B(1,a)$ and let $x_0$ its orthogonal projection onto the affine subspace $a + T_a A \subset \R^n$.  
Let $\gamma:[0,r_a(x)]\to \R^n$ be a unit speed line segment  joining $x_0 = \gamma(0)$ and $x= \gamma(r_a(x))$. 
Since $j^{k-1} \mathscr{T}_{a,k} F (x_0) = 0$, Taylor's theorem implies 
\begin{align*}
|  \mathscr{T}_{a,k} F(x)|
&\le
\frac{1}{k!} \cdot\big\|( \mathscr{T}_{a,k} F  \circ\gamma)^{(k)}\big\|_{C^0([0,r_a(x)])} \cdot  r_a(x)^k
\le
\frac{1}{k!} \cdot\|  \mathscr{T}_{a,k} F  \|_{C^k(B(1,a))}  \cdot  r_a(x)^k \, .  
\end{align*}
Hence we can work with $C_K := \tfrac{1}{k!} \cdot{\rm max}_{a \in K} \|  \mathscr{T}_{a,k} F \|_{C^k(B(1,a))}$.
\end{proof}

For the remainder term $ \mathscr{R}_{a,k}F (x) := F(x) - \mathscr{T}_{a,k} F(x)$ we have the following standard estimate:

\begin{lem} \label{lem:remainder} 
Let $K \subset A$ be compact.
Then  for $\delta > 0$ there exists $\eta > 0$ such that for all $0 < \eps < \eta$, $a \in K$, and $x \in B(2\eps,a)$ we have $x \in U$ and $|\mathscr{R}_{a, k} F(x) |\leq ( \delta^2 \eps)^{k}$. 
\end{lem} 

\begin{proof}
Let $r(x)=|x-a|$ denote the distance of $x$ to the point $a$.
The lemma follows from the standard estimate of the remainder term in the Taylor expansion:
\begin{equation*}
|\mathscr{R}_{a, k} F| = \mathrm{o}(r^k) = \mathrm{o}(\eps^k)
\end{equation*}
where the estimate is uniform on $K$.
\end{proof}

\section{Proof of Theorem~\ref{thm:FlexLem1}}

For $k=0$ the proof is easy and postponed to Appendix~\ref{sec-AnhangA}. 
In this section we will concentrate on the case $k \geq 1$. 

We use the ``tilde notation'' to denote by $\Ft$ the map $[0,1]\times U \to X$ corresponding to $F$ via $\Ft(t,u)=F(t)(u)$.
The condition $F\in C^\ell([0,1],C^k(U,X))$ is equivalent to the requirement that, in local coordinates $u^1,\ldots,u^n$ of $U$, the partial derivatives $(\frac{\partial}{\partial t})^m(\frac{\partial}{\partial u^1})^{\alpha_1}\cdots (\frac{\partial}{\partial u^n})^{\alpha_n}\Ft$ exist and are continuous for $m\le \ell$ and $|\alpha|=\alpha_1+\ldots+\alpha_n \le k$.
See \cite[Thm.~2]{F1945} for the case $\ell=k=0$ and \cite{AS2015} for the general case.

\begin{proof}[Proof of Theorem~\ref{thm:FlexLem1} for $k\ge1$]
\uline{\emph{Step~1:}}
We first show that we can assume without loss of generality that  $X \to V$ is a $C^{\infty}$-vector bundle. 

For this aim we equip the total space $X$ with an auxiliary complete Riemannian metric.
Let $T^{\rm vert}X\to X$ be the vertical tangent bundle, whose fibers are the tangent spaces of the fibers of $X$.
For each choice of  $f_0' \in C^{\infty}(V,X)$ we can consider the $C^\infty$-vector bundle $(f_0')^*T^{\rm vert}X \to V$.
The fiberwise exponential map yields a fiber-preserving $C^\infty$-diffeomorphism from a fiberwise convex open neighborhood $C \subset  (f_0')^*T^{\rm vert}X$ of the zero section onto an open neighborhood $W$ of $f_0'(V) \subset X$.

Choosing $f_0'$ close enough to $f_0$ in the strong topology on $C^0(  V, X)$ we can assume that the image of $f_0$ is contained in $W$.
Hence it defines a $C^{\kappa}$-section of the $C^{\infty}$-vector bundle $(f_0')^*T^{\rm vert}X \to X$.

Since $k \ge 1$ the function $F(t)|_{V_0}$ is independent of $t$ by assumption. 
Shrinking $U$ if necessary we can assume that the image of $\Ft$ is contained in $W \approx C$.
The constructions in Step~3 will never leave the image of $\Ft$, and also the mollifying procedure \eqref{eq:Mollie} in Step~4 happens within the fiberwise convex neighborhood $C$ of $0$. 

Hence we can and will assume without loss of generality that $X \to V$ is a $C^\infty$-vector bundle.

\medskip

\uline{\emph{Step~2:}}
We can furthermore assume that $V$ is an open subset of some Euclidean space $\R^n$ and $V_0$ is not only closed in $V$, but also closed in $\R^n$.

For this aim choose a smooth proper embedding $V \subset \R^n$ for the given smooth manifold $V$. 
By properness of the embedding $V$ and $V_0$ become closed subsets of $\R^n$. 
Furthermore  we find  a continuous map $\eta : V \to \R_+$ such that the normal exponential map along $V$ identifies $\{ (v, \xi) \mid v \in V,\, \xi \in (T_v V)^{\perp} \subset \R^n , \, | \xi | < \eta(v) \} \subset V \times \R^n$ with an open tubular neighborhood $N_V$ of $V$ in $\R^n$.
Let $p : N_V \to V$ be the bundle projection and $p^* X \to N_V$ be the pull back bundle. 
This results in a bundle map
\[ 
\xymatrix{
    p^* X  \ar[r]^{\overline{p}}  \ar[d] & X \ar[d]   \\
     N_V \ar[r]^{p}                      & V  \, .} 
\]
Then $p^{-1}(U) \subset N_V$ is an open neighborhood of the closed subset $V_0 \subset N_V$.
We define a PDR $\hat\RR\subset J^k(p^*X)$ over $N_V$ by setting 
$$
\hat\RR := \{j^k\phi(x) \mid x\in N_V, \, \phi \mbox{ is a local section of }p^*X\mbox{ near $x$ such that }\phi|_{\mathrm{dom}(\phi)\cap V}\mbox{ solves }\RR\}.
$$
One easily checks that this relation is open.
The section $p^*f_0$ (characterized by $\overline{p}\circ (p^*f_0) = f_0 \circ p$) of $p^*X$ solves $\hat\RR$. 
Similarly, the section $p^*F(t)$ of $p^*X|_{p^{-1}(U)}$ solves $\hat\RR$ over $p^{-1}(U)$ and the $(k-1)$-jet of $p^*F(t)$ is independent of $t$ along $V_0$. 
Conversely, any $C^{\kappa}$-solution of $\hat\RR$ (with $\kappa \geq k$) restricts to a $C^{\kappa}$-solution of $\RR$ over $V$. 

Hence, working with $N_V$ instead of $V$ and with the given $V_0 \subset V$, we can and will assume in the following that $V \subset \R^n$ is an open subset and $V_0$ is closed in $\R^n$. 

\medskip

\uline{\emph{Step~3:}}
Now  we prove Theorem~\ref{thm:FlexLem1}  for $\ell=\infty$ and $\kappa = k\ge1$.
Since we assume that $V \subset \R^n$ is an open subset we can work with global coordinates $(x^1, \ldots, x^n)$ on $V$. 

Since $V_0 \subset  \R^n$ is closed we find a countable family of compact sets $(K_{\nu})_{\nu\in \mathscr{I}}$, $K_{\nu} \subset V_0$, whose union covers $V_0$,  together with relatively compact open neighborhoods $K_\nu \subset U_{\nu} \subset U$  such that the family $( U_{\nu})_{\nu\in \mathscr{I}}$ satisfies:
\begin{enumerate}[\myicon]
\item This family is locally finite, that is,  around each point in $\R^n$ there is a neighborhood in $\R^n$ meeting only finitely many $U_\nu$; 
\item 
For each $\nu \in \mathscr{I}$ the set $\mathscr{I}_{\nu} := \{ \mu \in \mathscr{I} \mid U_{\nu} \cap U_{\mu} \neq \emptyset\}$ is finite;
\item
The vector bundle $X$ can be trivialized over each $\overline{U}_{\nu}$, and we fix such trivializations. 
\end{enumerate}

For each $\nu \in \mathscr{I}$ we fix $0 < \delta_{\nu} < \tfrac14$ and $0 < \Lambda_{\nu} < \tfrac12$. 
The precise values will be determined later. 
Put $\tilde\delta_\nu := \max\{\delta_\mu \mid \mu\in\mathscr{I}_\nu\}$.

Apply Lemma~\ref{lem:cutoffnonpure} to $A = V_0$, $ K = K_{\nu} \subset A$ and $U = U_{\nu}$. 
We get finite families $(\eps_{\nu,i})_{i \in I_{\nu}}$, $0 < \eps_{\nu, i} < \Lambda_{\nu}$, and $(a_{\nu,i})_{i \in I_{\nu}}$, $a_{\nu,i} \in K_{\nu}$, together with  $C^{\infty}$-functions $\rho_{\nu} : \R^n \to [0,1]$ as described in Lemma~\ref{lem:cutoffnonpure}. 

We define the $C^{\infty}$-function 
\begin{equation} \label{eq:deftau} 
   \tau := 1 - \prod_{\nu \in \mathscr{I}} ( 1 - \rho_{\nu}) : \R^n \to [0,1] \, .  
\end{equation} 
This is well defined and $C^{\infty}$ since $(U_\nu)_{\nu \in \mathscr{I}}$ is locally finite. 
Furthermore, ${\rm supp} (\tau) \subset \bigcup_{\nu\in\mathscr{I}}U_\nu \subset U$ and $\tau = 1$ on a neighborhood of $V_0$. 

Let $x \in {\rm supp}(\tau)$. 
By Lemma~\ref{lem:cutoffnonpure}~\eqref{cutoffnonpure-3} and finiteness of the $\I_\nu$ there are finitely many $(\nu, i)$ with $x \in B(2\eps_{\nu,i}, a_{\nu,i})$.
Amongst all such $(\nu, i)$ we choose the one for which $\Omega_{\delta_{\nu}^2\eps_{\nu, i} , a_i} (x)$ attains its minimal value. 
The estimate in Lemma~\ref{lem:cutoffnonpure}~\eqref{cutoffnonpure-3} and the product rule yield
\begin{equation}
|D^\alpha\tau(x)|
\leq
C_{k,\nu} \cdot \Omega_{\delta_{\nu}^2\eps_{\nu, i} , a_{\nu , i}} (x)^{-|\alpha|} \cdot | \ln \tilde\delta_{\nu}|^{-1}
\label{eq:Datau}
\end{equation}
for all $\alpha$ with $1\le|\alpha|\le k$.
Here the constant $C_{k,\nu}$ is independent of $\delta_\mu$, $\Lambda_\mu$, $\eps_{\mu, i}$, $a_{\mu, i}$, and $x$.

We define sections $f(t) : V \to X$ by 
\begin{equation} 
f(t)(v) 
:= 
\begin{cases}
\tilde F(t\tau(v), v), & \text{if }v\in U,\\
f_0(v), & \text{else}.
\end{cases}
\label{eq:deff} 
\end{equation}
This defines $C^k$-sections of $X$ which depend smoothly on $t$.
We will show that the $f(t)$ solve $\RR$ if the constants $\delta_\nu$ and $\Lambda_\nu$ are properly chosen.

Using the trivialization of the bundle $X$ over $U_\nu$ we identify sections of $U_\nu$ with vector-valued functions.
Induction\footnote{For $|\alpha|=0$ equation~\eqref{eq:Errorterm1} is nothing but the definition $\ft(t,x)=\Ft(t\tau,x)$. The induction step consists of differentiating \eqref{eq:Errorterm1}.} 
on $|\alpha|$ shows
\begin{equation}
(D^{\alpha}f)(t,x)
=
(D^{\alpha}\Ft)(t\tau(x) ,x) 
+ 
\sum_{\lambda=1}^{|\alpha|} \sum_{\hat\alpha} (D^{\hat\alpha} \partial_t^\lambda \Ft)(t\tau(x),x) \cdot t^\lambda \cdot P_{\alpha\hat\alpha \lambda}(x).
\label{eq:Errorterm1}
\end{equation}
Since $F(t)$ solves $\RR$ and the relation is open there exists $\Delta_\nu>0$ such that if 
\begin{equation}
\bigg| \sum_{\lambda=1}^{|\alpha|} \sum_{\hat\alpha} (D^{\hat\alpha} \partial_t^\lambda \Ft)(t\tau(y),y) \cdot t^\lambda \cdot P_{\alpha\hat\alpha \lambda}(y)\bigg| 
< 
\Delta_\nu
\label{eq:RRest}
\end{equation}
for all $|\alpha|\le k$ and for some $y\in U_\nu$ then $f(t)$ solves $\RR$ over $y$ as well.

In \eqref{eq:Errorterm1} and \eqref{eq:RRest} the inner sum is taken over all multiindices $\hat\alpha$ satisfying $\hat\alpha\le\alpha$ and $\lambda+|\hat\alpha| \le|\alpha|$.
Each $P_{\alpha\hat\alpha  \lambda}$ is a universal polynomial in derivatives of $\tau$, independent of $F$.
It is weighted homogeneous of degree $|\alpha|-|\hat\alpha|$ if we assign to each $a^\mathrm{th}$ $x$-derivative  of $\tau$ the weight $a$.
The product rule together with \eqref{eq:Datau} shows, estimating higher powers of $|\ln \tilde\delta_{\nu}|^{-1}$ by $| \ln \tilde\delta_{\nu}|^{-1}$,
\begin{equation} \label{eq:estP} 
|P_{\alpha\hat\alpha  \lambda}(x)| 
\leq 
C_{k,\nu} \cdot \Omega_{\delta_{\nu}^2\eps_{\nu, i} , a_{\nu , i}} (x)^{|\hat\alpha|  - |\alpha|} \cdot | \ln \tilde\delta_{\nu}|^{-1} \, . 
\end{equation}
Note that in \eqref{eq:Errorterm1} we have $\lambda + |\hat\alpha| \leq |\alpha|$ and $\lambda \geq 1$, thus the exponent of $\Omega_{\delta_{\nu}^2\eps_{\nu, i} , a_{\nu , i}} (x)$ in \eqref{eq:estP} is negative.

Lemma~\ref{lem:taylor} implies that there is a constant $C_{K_\nu}>0$, independent of $x$, $\delta_\mu$, and $\Lambda_\mu$, such that for our $x \in B(2\eps_{\nu,i}, a_{\nu,i})\subset B(1, a_{\nu,i})$,  $0 \leq |\hat\alpha| \leq k$, $t \in [0,1]$, and $0 < \ell \leq k$ we have  
\begin{equation}
|  \mathscr{T}_{a_{\nu, i}, k-|\hat\alpha|} ( D^{\hat\alpha} \partial_t^{\ell}\Ft )  (t,x) | 
\leq 
C_{K_{\nu}}  \cdot r_{a_{\nu, i}}(x)^{k - |\hat\alpha|} \, . 
\label{eq:TaylorAbsch}
\end{equation}
Moreover, by Lemma~\ref{lem:remainder}, we can choose $\Lambda_{\nu}$ so small (depending on $\delta_\nu$ but independently of $x$) that 
\begin{equation}
|\mathscr{R}_{a_{\nu, i}, k-|\hat\alpha|}  (D^{\hat\alpha} \partial_t^{\ell}\Ft )  (t,x) | 
\leq 
( \delta_{\nu}^2 \eps_{\nu,i})^{k - |\hat\alpha|} \, . 
\label{eq:RestAbsch}
\end{equation}
With this  choice of  $\Lambda_{\nu}$ we get, using \eqref{eq:estP}, \eqref{eq:TaylorAbsch} and \eqref{eq:RestAbsch},
\begin{align} 
|(D^{\hat\alpha} \partial_t^\lambda \Ft)&(t\tau (x),x) \cdot t^\lambda \cdot P_{\alpha\hat\alpha \lambda}(x)|  \notag\\
&\leq   
C_{k,\nu} \cdot |(D^{\hat\alpha} \partial_t^\lambda \Ft)(t\tau (x),x)|\cdot \Omega_{\delta_{\nu}^2\eps_{\nu, i} , a_i} (x)^{|\hat\alpha|  - |\alpha|} \cdot | \ln \tilde\delta_{\nu}|^{-1}  \notag\\
&\leq   
C_{k,\nu} \cdot  \Big( C_{K_{\nu}} \cdot r_{a_{\nu,i}}(x)^{k - |\hat\alpha|}  +  ( \delta_{\nu}^2 \eps_{\nu, i} )^{k - |\hat \alpha|} \Big) \cdot \Omega_{\delta_{\nu}^2\eps_{\nu, i} , a_i} (x)^{|\hat\alpha|  - |\alpha|}\cdot |\ln \tilde\delta_{\nu}|^{-1} \notag\\
&\leq   
C_{k,\nu} \cdot  \Big( C_{K_{\nu}} \cdot r_{a_{\nu,i}}(x)^{k - |\hat\alpha|}  \cdot  r_{a_{\nu,i}}(x)^{ |\hat \alpha| - |\alpha|} +  ( \delta_{\nu}^2 \eps_{\nu, i} )^{k - |\hat \alpha|} \cdot ( \delta_{\nu}^2 \eps_{\nu, i})^{ | \hat \alpha| - |\alpha|}\Big) \cdot |\ln \tilde\delta_{\nu}|^{-1} \notag\\
&=  
C_{k,\nu} \cdot  \Big( C_{K_{\nu}} \cdot r_{a_{\nu,i}}(x)^{k - |\alpha|}  +  ( \delta_{\nu}^2 \eps_{\nu, i} )^{k - |\alpha|} \Big) \cdot |\ln \tilde\delta_{\nu}|^{-1} \notag\\
&\leq  
C_{k,\nu} \cdot (C_{K_{\nu}}+1)  \cdot  | \ln \tilde\delta_{\nu}|^{-1}    \, . 
\label{eq:glorious}  
\end{align}
Thus for $\tilde\delta_{\nu}>0$ sufficiently small estimate \eqref{eq:RRest} holds for $y=x$.
This imposes finitely many conditions on each $\delta_\nu$ and can therefore be arranged.
This completes the proof of Theorem~\ref{thm:FlexLem1} for $\ell=\infty$ and $\kappa = k\ge1$.

\medskip 

\uline{\emph{Step~4:}}
Now we drop the differentiability assumption in the path variable and consider the case $\ell=0$ and $\kappa = k\ge 1$.
We equip the vector bundle $X\to V$ with a Euclidean structure and a compatible $C^{\infty}$-connection $\nX$. 
We introduce a second path variable and define
\begin{equation*}
\F(s,t) := 
\begin{cases}
F(0), & \mbox{ for } t\le 0; \\
F(st), & \mbox{ for } 0\le t\le 1; \\
F(s), & \mbox{ for } t\ge 1.
\end{cases}
\end{equation*}
Then $\F\in C^0([0,1]\times \R,C^k(U,X))$.

Let $\chi\in C^\infty(\R)$ be nonnegative with $\mathrm{supp}(\chi) \subset [-1,1]$ and $\int_{-1}^1 \chi(\sigma)\, d\sigma = 1$.
We mollify $\F$ in the $t$-variable by putting, for $\delta\in(0,1]$,
\begin{equation}
\F_s^\delta(t) :=
\frac1\delta \int_\R \chi\bigg(\frac{(1+2\delta)t -\delta -\sigma}{\delta}\bigg)\F(s,\sigma)\, d\sigma.
\label{eq:Mollie}
\end{equation}
Then each $\F_s^\delta\in C^\infty([0,1], C^k(U,X))$ and $\F_s^\delta$ is also smooth in $\delta$.
Using $\frac1\delta \int_\R \chi\big(\frac{(1+2\delta)t -\delta -\sigma}{\delta}\big)\, d\sigma=1$ and the support property of $\chi$ it is straightforward to check
\begin{align*}
\F_s^\delta(0) &= F(0), \\
\F_s^\delta(1) &= F(s), \\
\F_0^\delta(t) &= F(0),
\end{align*}
for all $s,t\in [0,1]$.
Moreover, since $\F$ is uniformly continuous, we can, given $\eps>0$ and a compact subset $K\subset U$, find an $\eps'>0$ such that $|\F(s,t_1)|_K-\F(s,t_2)|_K|<\eps$ for all $s$ and all $t_1,t_2$ with $|t_1-t_2|<\eps'$.
Then
$$
\big|\F_s^\delta(t)|_K - \F(s,t)|_K\big| 
\le
\frac1\delta \int_\R \chi\bigg(\frac{(1+2\delta)t -\delta -\sigma}{\delta}\bigg)\big|\F(s,\sigma)|_K-\F(s,t)|_K\big|\, d\sigma
<
\eps
$$
provided $2\delta<\eps'$.
The same argument applies to the covariant derivatives.
Thus $\F_s^\delta(t)$ converges in $C^k(U,X)$ to $\F(s,t)$ uniformly in $s$ and $t$ as $\delta\to 0$ with respect to the weak $C^k$-topology.
In particular, $\F_s^\delta$ solves $\RR$ over $K$ for $\delta\le\delta(K)$.

We choose a locally finite cover of $U$ by relatively compact open sets $O_\nu$.
Then we can find a positive smooth function $\delta:U\to \R$ such that $\delta(v) \le \delta(\overline O_\nu)$ for all $v\in\overline O_\nu$ and all $\nu$.

We rewrite \eqref{eq:Mollie} as
\[
\F_s^\delta(t) 
=
\frac1\delta \int_\R \chi\bigg(\frac{(1+2\delta)t -\delta -\sigma}{\delta}\bigg)(\F(s,\sigma)-f_0)\, d\sigma + f_0
\]
and recall that all derivatives up to order $k-1$ of $\F(s,\sigma)-f_0$ vanish along $V_0$.
Thus all derivatives up to order $k-1$ of $\frac1\delta \int_\R \chi\bigg(\frac{(1+2\delta)t -\delta -\sigma}{\delta}\bigg)(\F(s,\sigma)-f_0)\, d\sigma$ vanish as well.
For the $k^\mathrm{th}$ derivative we obtain
\begin{align*}
\nX^{(k)} \F_s^{\delta(\cdot)}(t) 
&= 
\frac1\delta \int_\R \chi\bigg(\frac{(1+2\delta)t -\delta -\sigma}{\delta}\bigg)\nX^{(k)}(\F(s,\sigma)-f_0)\, d\sigma + \nX^{(k)}f_0 \\
&=
\frac1\delta \int_\R \chi\bigg(\frac{(1+2\delta)t -\delta -\sigma}{\delta}\bigg)\nX^{(k)}\F(s,\sigma)\, d\sigma .
\end{align*}
No derivatives of $\delta$ occur in this formula.
In particular, $v\mapsto \F_s^{\delta(v)}(t)(v)$ solves $\RR$ along $V_0$.
By shrinking $U$ if necessary, $v\mapsto \F_s^{\delta(v)}(t)(v)$ solves $\RR$ also over $U$.

We can now apply the results obtained in Step~3 to each $\F_s^{\delta(\cdot)}$ and get $f_s\in C^{\infty}([0,1],C^k(V  ,X))$ such that for all $s,t\in[0,1]$
\begin{enumerate}[\myicon]
\item 
each $f_s(t)$ is a section of $X$ solving $\RR$;
\item
$f_s(0)=f_0$;
\item
$f_s(t)|_{U_0} = \F_s^{\delta(\cdot)}(t)|_{U_0}$;
\item
$f_s(t)|_{V\setminus U} = f_0|_{V\setminus U}$.
\end{enumerate}
Furthermore, over $U$ the section $f_s(t)$ is of the form $\widetilde{f_s}(t,v)=\widetilde{\F_s^{\delta(v)}}(t\tau(v),v)$ with a $\tau$ as in Step~3. 
Note that here we can in fact choose $\tau$  and $U_0$ independently of~$s$.

We set $f(s) := f_s(1)$. 
Then $f \in C^0([0,1], C^k(V,X))$ and 
\begin{enumerate}[\myicon]
\item 
$f(s)$ is a section of $X$ solving $\RR$;
\item
Over $U$ we have $f(0) = f_0(1) = \F_0^{\delta(\cdot)}(\tau(\cdot)) = F(0) = f_0$;
\item
Over $U_0$ we have $f(s) = f_s(1) = \F_s^{\delta(\cdot)}(1) = F(s)$;
\item
Over $V\setminus U$ we have $f(s) = f_s(1) = f_0$.
\end{enumerate}

\medskip

\uline{\emph{Step~5:}}
Finally, if $F\in C^\ell([0,1],C^\kappa(U,X))$ then $\F_s^{\delta(v)}(t)(v)$ is $C^\ell$ in $s$, smooth in $t$ and $C^\kappa$ in~$v$.
Hence $f\in C^\ell([0,1],C^\kappa(U,X))$.
This concludes the proof of Theorem~\ref{thm:FlexLem1}. 
\end{proof}

\begin{rem}
One cannot drop the assumption that $V_0$ is a closed subset even if it is a smooth embedded submanifold.
For example, we may choose $V=\R^2$ and $V_0 = \{(x,y)\in\R^2 \mid x=0, y>0\}$.
As an open neighborhood of $V_0$ we choose $U=\{(x,y)\in\R^2 \mid y>0\}$.
The fiber bundle $X$ is the trivial real line bundle; so our sections are just real-valued functions.
We consider the case $k=1$ and the relation $\RR=J^1X$.
In other words, the relation does not impose any restrictions on our functions.

Let $f_0\equiv0$ on $V$ and $\Ft(t,x,y)=tx\sin(1/y)$ on $U$.
The assumptions of Theorem~\ref{thm:FlexLem1} are now satisfied (except for closedness of $V_0$) but for $t>0$ the derivative $\frac{\partial F}{\partial x}$ does not have a limit as $y\to 0$. 
Thus no restriction of $F$ to $[0,1]\times U_0$ for any neighborhood $U_0$ of $V_0$ can be extended as a $C^1$-map to $[0,1]\times V$.
\end{rem}

\begin{rem} 
The assumption $j^{k-1}F(t)=j^{k-1}f_0$ along $V_0$ cannot be dropped either. 
For example, let $V = \R$ and $V_0 = \{ -1, +1\}$. 
We still work with real-valued functions and $k=1$. 
Let the relation $\RR$ not impose any restrictions on $0$-jets and force first derivatives to lie in the interval $(-1, 1)$. 
Let $f_0\equiv 0$,  $U = \R \setminus \{0\}$ and set $\Ft(t,x) = 10 \cdot t$ for $x>0$ and $\Ft(t,x) = - 10 \cdot t$ for $x<0$. 
With these choices a function $f$ as in Definition~\ref{def:GromovFlexible} does not exist. 
\end{rem}

\begin{rem}\label{rem:Values}
If $F$ is sufficiently regular in the path variable, more precisely if $F\in C^{\ell}([0,1],C^k(U,X))$ with $\ell \geq k$, then, as in the proof of Theorem~\ref{thm:FlexLem1}, we can use the ansatz $f(t)(v)=\Ft(t\tau(v),v)$ for $v 
\in U$ to obtain $f \in C^{\ell - k}([0,1], C^k(V,X))$. 
With this definition the deformation $f$ takes only values that are taken by $f_0$ and $F$.
This means that for the values (but not their derivatives) we can also preserve nonopen relations.
For instance, if our sections are real-valued functions and $f_0\ge 0$ and $F\ge 0$ holds then we have also $f\ge 0$.
\end{rem}

We get the following family version of Theorem~\ref{thm:FlexLem1}.

\begin{addendum} \label{add:family} 
Let $K$ be a compact Hausdorff space and let  $k\in \N_0$. 
Let $f_0\in C^0(K,C^k(V,X))$ and let $F\in C^0(K,C^0([0,1],C^k(U,X)))$ such that $f_0(\xi)$ and $F(\xi)$ fall in Setting~\ref{setting} for each $\xi\in K$. 

Then parametrized local flexibility holds: 
There exists $f\in C^0(K, C^0([0,1], C^k(V,X)))$ such that $f(\xi)$ enjoys the properties of Definition~\ref{def:GromovFlexible} for each $\xi$ with $U_0$ independent of $\xi$. 

Moreover, let $\ell\in\{0,1,\ldots,\infty\}$, $\kappa\in\{k,k+1,\ldots,\infty\}$. 
Let $f_0\in C^0(K,C^\kappa(V,X))$ and $F\in C^0(K,C^\ell([0,1],C^\kappa(U,X)))$.
Then we can assume in addition that $f \in C^0(K, C^\ell([0,1], C^\kappa(V,X)))$.

Finally, if $\xi \in K$ is such that the deformation $F(\xi)$ is constant in the path variable, then in all the previous cases $f(\xi)$ can be assumed to be constant in the path variable as well. 
\end{addendum}

\begin{proof} We concentrate on the case $k \geq 1$ and leave the case $k = 0$ to the reader. 

If $F(\xi)  \in C^{\infty}([0,1], C^k(U,X))$ as in Step~3 of the proof of Theorem~\ref{thm:FlexLem1} then the bounds on the $\Lambda_\nu$ and $\delta_{\nu}$ depend on bounds on derivatives of $F$.
Thus they can be chosen independently of $\xi\in K$, by compactness of $K$.
Therefore the cutoff function $\tau$ in \eqref{eq:deftau} can be chosen independently of $\xi$.
Hence $f$ depends continuously on $\xi$.

If $F(\xi) \in C^0([0,1], C^k(U,X)) $  then the function $\delta$ in Step~4 of the proof can be chosen independently of $\xi$, again by compactness of $K$.
The mollifying procedure in \eqref{eq:Mollie} yields a continuous map $C^0([0,1], C^0(\R, C^k(U,X))) \to C^0([0,1], C^{\infty} (\R, C^k(U,X)))$.
Then Step~3 applies.
A similar argument applies to the case $F\in C^0(K,C^\ell([0,1],C^\kappa(U,X)))$ for more general $\ell$ and $\kappa$. 

The last assertion follows directly from the definition of $f(\xi)(t)$ in \eqref{eq:deff} if $F(\xi) \in C^{\infty}([0,1], C^k(U,X))$. 
In the remaining cases we observe that the mollified function $\F_s^{\delta}(t)$ in \eqref{eq:Mollie} is constant in $s$ and $t$ if the original function $F(t)$ is constant in $t$. 
 \end{proof}

This can be reformulated in homotopy theoretic language. 
Let $\phi$ be a fixed $C^k$-germ of sections of $X$ around $V_0$ solving $\RR$. 
We say that a $C^k$-section of $X$ over some open neighborhood of $V_0$  is {\em $\phi$-compatible} if it has the same $(k-1)$-jet along $V_0$ as $\phi$. 
Now consider 
\begin{enumerate}[\myicon] 
   \item the space $E$ of all $\phi$-compatible $C^k$-solutions of $\RR$ over $V$, 
   \item the space $E_0$ of all $\phi$-compatible $C^k$-germs of solutions of $\RR$ around $V_0$. 
\end{enumerate}
The space $E_0$ is equipped with the quasi-$C^k$-topology induced by the directed system $C^k(U, X)$, $V_0 \subset U \subset V$ open. 
This means  that a continuous map $K \to E_0$ for compact $K$ is represented by a continuous map $K \to C^k(U, X)$ for some open $V_0 \subset U \subset V$ with image in the $\phi$-compatible solutions of $\RR$ over $U$.
 
Applying Addendum~\ref{add:family} we now have the following assertion.

\begin{cor} \label{cor:homotop} 
The restriction map $E \to E_0$ has the homotopy lifting property with respect to all compact Hausdorff spaces. 
In particular, it is a Serre fibration.
\hfill $\Box$
\end{cor} 

This formulation provides a link of local flexibility to other h-principle concepts, such as  flexibility and microflexibility, compare \cite[Section~1.4.2~(B')]{Gromov86}.

\section{Applications}

The following applications from different mathematical contexts illustrate situations in which local flexibility applies naturally.

\subsection{Deforming hypersurfaces} \label{sec:defhypersurf} 

Let $V$ be an $n$-dimensional manifold and $f_0:V \looparrowright \R^{n+1}$ an immersion.
For a constant $\mu>0$ we call $f_0$ \emph{$\mu$-convex} if all  eigenvalues of the Weingarten map (the principal curvatures) of $f_0(V)$ w.r.t.\ suitable choice of unit normal are $\ge\mu$ everywhere.
If $f_0$ is a $\mu$-convex embedding then $f_0(V)\subset B$ for any closed ball $B$ of radius $\frac1\mu$ whose boundary touches $f_0(V)$ tangentially at a point $p$ and is curved in the same direction, i.e.\ the unit normals of $f_0(V)$ and $\partial B$ at $p$ coincide when chosen such that both Weingarten maps are positive, see \cite[Section~6.3]{E1986}.

\begin{center}
\includegraphics[scale=0.5]{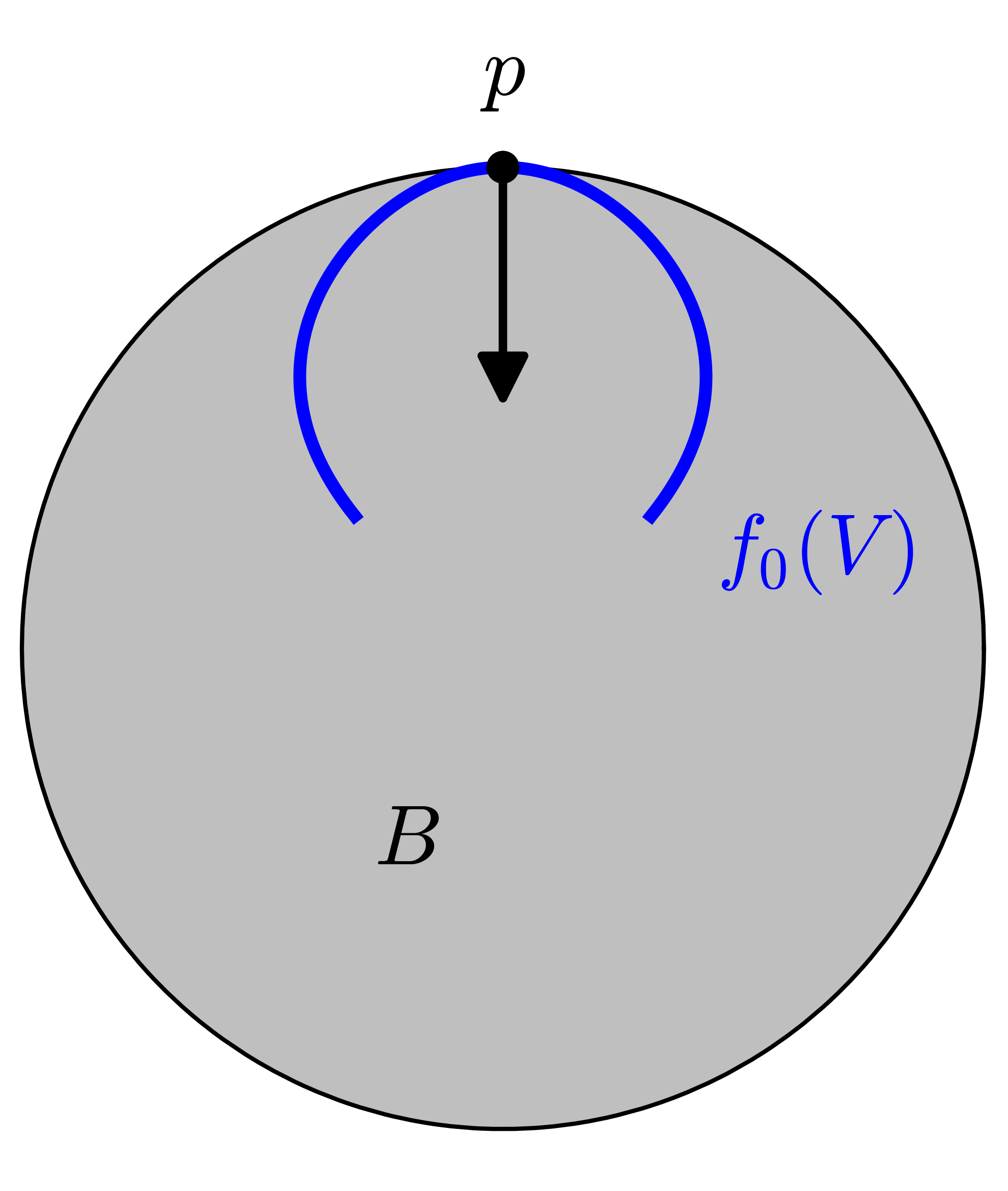}

\Abb{Image $f_0(V)$ is $\mu$-convex at $p$}
\vspace{2mm}
\end{center}

Now let $f_0:S^n \hookrightarrow \R^{n+1}$ be the standard embedding.
Then $f_0$ is $1$-convex and rigid in the following sense:
Given a point $p\in S^n$ we cannot find an embedding $f_1$ such that $f_1=f_0$ on the opposite hemisphere $S^n_{p,-}:=\{v\in S^n \mid \langle v,p \rangle\le 0\}$, $f_1$ is $1$-convex everywhere and $\mu$-convex near $p$ for some $\mu>1$.

Namely, assume such an $f_1$ exists.
By $1$-convexity and since $S^n_{p,-}$ is contained in $f_1(S^n)$ we have $f_1(S^n) \subset \bar B_1(0)$.
By $\mu$-convexity near $f_1(p)$, $f_1(S^n)$ contains points in the interior of $\bar B_1(0)$.
Let $q\in f_1(S^n)$ be such a point.
Again by $1$-convexity, applied at $q$, $f_1(S^n)\subset \bar B_1(m)$ for some $m\neq0$.

\begin{center}
\includegraphics[scale=0.5]{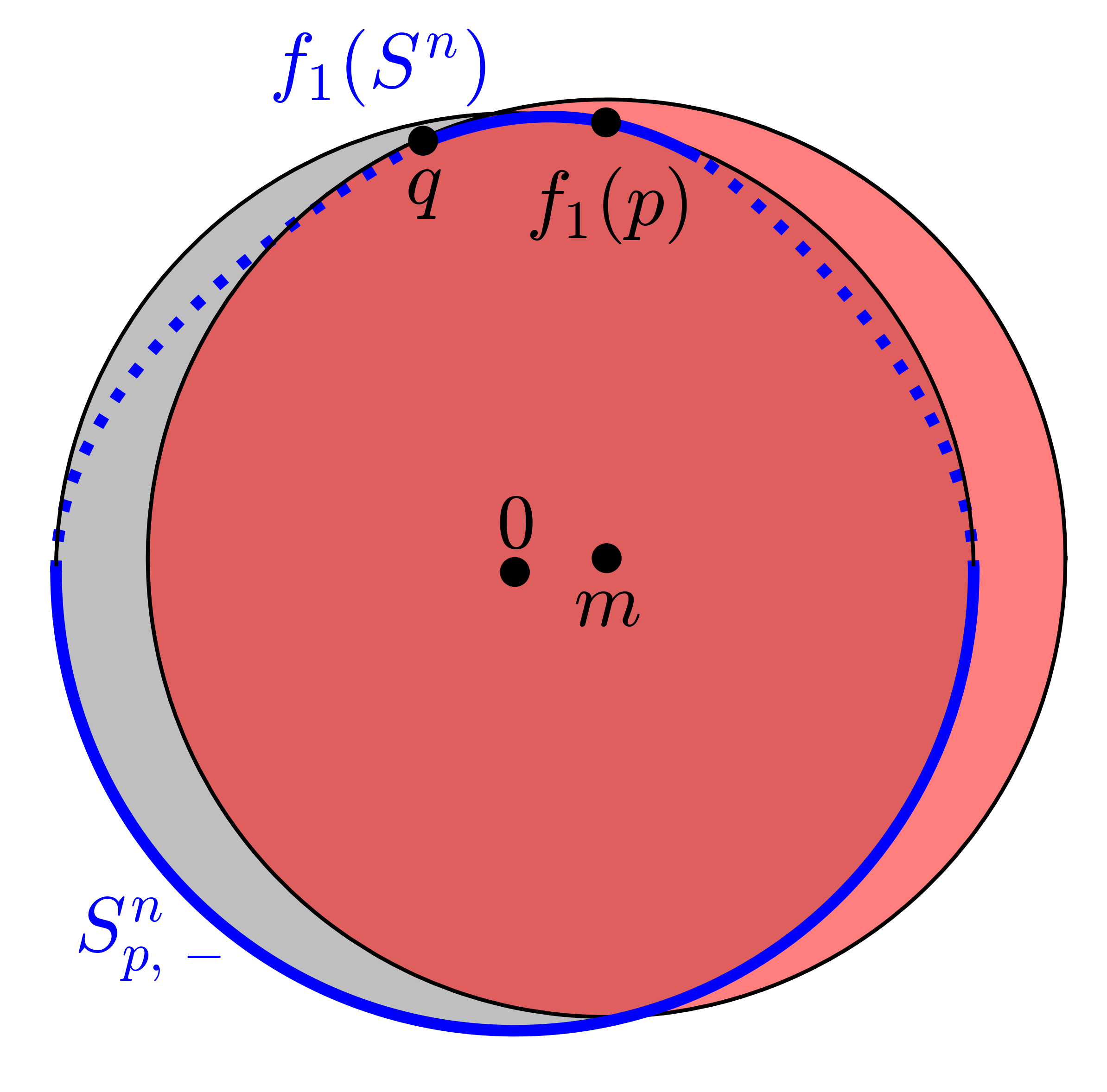}

\Abb{Image $f_1(S^n)$ is $\mu$-convex at $f_1(p)$}
\vspace{2mm}
\end{center}

On the other hand, $S^n_{p,-} \subset f_1(S^n)\subset \bar B_1(m)$, which is possible only if $m=0$.
We have arrived at a contradiction.

Now we relax the conditions.
Let $\eps>0$.
We look for an embedding $f_1:S^n\hookrightarrow \R^{n+1}$ such that $f_1=f_0$ on $S^n_{p,-}$, $f_1$ is $(1-\eps)$-convex everywhere and $\mu$-convex near $p$ for any given $\mu>1$.
Such an $f_1$ actually exists.

To see this we put $V=S^n$, $V_0=\{p\}$ and $U=S^n\setminus S^n_{p,-}$.
Being an immersion with principal curvatures $>1-\eps$ imposes an open partial differential relation $\RR$ of order $2$ on $C^\infty(V,\R^{n+1})$.
Consider a smooth $1$-parameter family of smooth diffeomorphisms $\Psi(t):U\to U$ such that $\Psi(t)(p)=p$ and $d\Psi(t)|_p=\frac{\mu}{\mu-(\mu-1)t}\id_{T_pU}$, $t\in [0,1]$.
We set 
$$
\tilde F(t,v):= \bigg(1-\frac{\mu-1}{\mu}t\bigg)\Psi(t)(v)+\frac{\mu-1}{\mu}tp .
$$
Then each $F(t)$ is a $\frac{\mu}{\mu-(\mu-1)t}$-convex embedding $U\hookrightarrow \R^{n+1}$ satisfying $F(t)(p)=p$ and $dF(t)|_p=\id_{T_pU}$ for all $t$.
Hence the $1$-jet of $F(t)$ is constant at $V_0=\{p\}$.

We apply Theorem~\ref{thm:FlexLem1} and obtain $f\in C^\infty([0,1],C^\infty(S^n,\R^{n+1}))$ such that each $f(t)$ is a $(1-\eps)$-convex immersion,   $f(t)=f_0$ on $S^n_{p,-}$ and $f(t)=F(t)$ near $p$.
In particular, $f_1 := f(1)$ is $\mu$-convex near $p$.

It remains to see that $f_1$ is injective and hence an embedding since $V$ is compact.
By Remark~\ref{rem:Values}, $f_1$ is of the form $f_1(v)=\tilde F(\tau(v),v)$ on $U$ for some function $\tau:U\to[0,1]$.
Now injectivity of $f_1$ follows from injectivity of $\tilde F|_{[0,1]\times (U\setminus\{p\})}$.

\subsection{Deforming differential forms}

Let $V$ be a smooth manifold and let $V_0 \subset V$ be a smooth submanifold, which is closed as a subset. 
For $p \geq 1$ we consider the exterior form bundle  $\Lambda^p (T^*V) \to V$ and fix an open subset   $\QQ \subset \Lambda^p (T^*V)$. 
Let $\omega_0  \in C^1( V , \Lambda^p (T^*V)) $ be a closed differential form of degree $p$ solving $\QQ$. 

Let $U$ be an open neighborhood of $V_0$ in $V$ and let $\Omega  \in C^{0} ([0,1], C^1(U,\Lambda^p (T^*V)))$ be a path of closed forms solving  $\QQ$ with $\omega_0  |_{U} = \Omega(0)$. 
W.l.o.g.\ we can assume that $U$ is a tubular neighborhood, shrinking it if necessary.
Furthermore, we assume that $i_0^*\Omega(t)$ is constant in $t$ where $i_0:V_0\hookrightarrow U$ is the embedding.
Note that this is a weaker condition than the restriction of $\Omega(t)$ to $V_0$ being constant in $t$. 

Now $\Omega(t)-\Omega(0)$ is a family of closed forms satisfying $i_0^*(\Omega(t)-\Omega(0))=0$.
Let $\rho\in C^\infty([0,1],C^\infty(U,U))$ be a retraction of $U$ onto $V_0$, i.e.\ $\rho(1)=\id_U$, $\rho(s)|_{V_0}=\id_{V_0}$ for all $s\in[0,1]$ and $\rho(0)(U)=V_0$.
We obtain $F \in C^{0}([0,1] , C^1(U,\Lambda^{p-1}(T^*V)))$ by setting
$$
F(t)(\xi_1,\ldots,\xi_{p-1}) := \int_0^1 (\Omega(t)-\Omega(0))\big(\tfrac{d\rho}{ds}(s),\rho(s)_*\xi_1,\ldots,\rho(s)_*\xi_{p-1}\big)\, ds 
$$
where $\xi_1,\ldots,\xi_{p-1}\in T_uU$ and $u\in U$.
Then $F$ has the following properties:
\begin{enumerate} [\myicon] 
\item $F(0) \equiv 0$; 
\item $F(t)|_{V_0} \equiv 0$ for all $t\in[0,1]$;
\item $\Omega(t) = \Omega(0) + d F(t)$ for all $t\in[0,1]$;
\item If $\ell \geq 0$, $\kappa \geq 1$, and $\Omega\in C^\ell ([0,1],C^\kappa(U,\Lambda^p(T^*V)))$, then $F \in C^{\ell }([0,1] , C^\kappa(U,\Lambda^{p-1}(T^*V)))$.
\end{enumerate} 
Only the third property requires a small computation, compare e.g.\ \cite[Prop.~6.8]{Silva}.

We set $X := \Lambda^{p-1}(T^* V) \to V$ and $k := 1$.
The condition $\omega_0 + d\eta \in \QQ$ imposes an open relation on the $1$-jet of the $(p-1)$-form $\eta$.
Denote this first-order relation on $\Lambda^{p-1}(T^*V)$ by $\RR$.
Now we apply Theorem~\ref{thm:FlexLem1} to $f_0=0$, $F(t)$ and $\RR$ and obtain $f\in C^0([0,1],C^1(V,\Lambda^{p-1}(T^*V)))$ such that $f(0)=f_0=0$ and $f(t)=F(t)$ on a smaller neighborhood $U_0$ of $V_0$.

For simplicity, let us assume that $\omega_0$ and $\tilde\Omega$ are smooth.
Then $\tilde F$ is smooth and so is $\tilde f$.
We obtain a smooth family of closed smooth $p$-forms
\[
     \omega(t) := \omega_0 + df(t) 
\]
solving $\QQ$, coinciding with $\Omega(t)$ on $U_0$, and $\Omega(t)=\omega_0$ outside $U$.
We summarize:

\begin{prop}\label{prop:closedform}
Let $V$ be a smooth manifold, $V_0\subset V$ a submanifold, closed as a subset.
Let $U$ be an open neighborhood of $V_0$ in $V$.
Let $\QQ\subset\Lambda^p (T^*V)$ be an open subset.

Let $\omega_0$ be a smooth closed differential $p$-form on $V$ solving $\QQ$, $p\ge1$.
Let $\Omega  \in C^{\infty} ([0,1], C^\infty(U,\Lambda^p (T^*V)))$ such that $\Omega(t)$ is closed and solves $\QQ$, $\Omega(0)=\omega_0$ on $U$ and $i_0^*\Omega(t)$ is constant in $t$. 

Then there exists a smaller neighborhood $U_0$ of $V_0$ and $\omega\in C^{\infty} ([0,1], C^\infty(V,\Lambda^p (T^*V)))$ such that each $\omega(t)$ is closed and solves $\QQ$, $\omega(0)=\omega_0$, $\omega(t)=\Omega(t)$ on $U_0$ and $\omega(t)=\omega_0$ outside $U$.
Moreover, the de Rham cohomology class of $\omega(t)$ is independent of $t$.
\hfill$\Box$
\end{prop}

The following table lists some geometric examples in which Proposition~\ref{prop:closedform} applies:

\renewcommand{\arraystretch}{1.2}
\begin{center}
\begin{tabulary}{\textwidth}{|C|C|C|C|}
\hline
$V$ & form degree~$p$ & condition $\QQ$ & resulting geometric structure \\
\hline
\hline
$n$-manifold & $1$ & nonvanishing &  codimension-$1$-foliation \\
\hline
$2n$-manifold & $2$ & $\omega^n\neq0$ & symplectic structure\\
\hline
$7$-manifold & $3$ & definite & closed $G_2$-structure \\
\hline
\end{tabulary}

\smallskip
\Tab{Geometric structures defined by a closed form satisfying an open condition}
\end{center}
\renewcommand{\arraystretch}{1}

The second example is a well-known consequence of the relative Moser lemma, see \cite[Theorem 7.4]{Silva}.
A similar argument can be applied to  the first example, where we use the one-to-one correspondence (after the choice of a Riemannian metric on $V$) of nonvanishing $1$-forms and nonvanishing vector fields in order to solve the Moser equation. 
For background on the $G_2$-example we refer to~\cite[Sections~3.1 and 4.6]{Bryant}. 
In this case a Moser type argument cannot be applied because $G_2$-structures induce Riemannian metrics which have local invariants.

Dropping the closedness conditions on the forms, Theorem~\ref{thm:FlexLem1} can be applied directly to $\Omega$ for $k = 0$, without passing to a family  $F(t)$ of primitives of $\Omega(t) - \Omega(0)$. 
This can be used to extend local deformations of contact forms, for instance.

\subsection{Prescribing the lapse function}

This application deals with Lorentzian geometry.
For a general introduction to and standard notation in this field see e.g.~\cite{BEE96}.
Let $V$ be a time-oriented globally hyperbolic Lorentzian manifold.
Then $V$ is isometric to $\R\times\Sigma$ with metric 
\begin{equation}
g=-N^2dT^2+g_T
\label{eq:BSMetric}
\end{equation}
where $N:V\to\R$ is smooth and positive, $T:V\to\R$ is smooth with past-directed timelike gradient such that each level $\{T_0\}\times\Sigma$ is a Cauchy hypersurface, see \cite[Thm.~1.1]{BS2005}.
The levels are then automatically closed, smooth, spacelike hypersurfaces.
Here $g_T$ is a smooth $1$-parameter family of Riemannian metrics on the levels.
We will call $N$ the \emph{lapse function} and $T$ the \emph{Cauchy time function}.
A simple computation shows $g(\grad T,\grad T)=-N^{-2}$.

In \cite[Thm.~1.2]{BS2006} Bernal and S\'anchez show that one can prescribe the Cauchy hypersurface.
More precisely, let $V_0$ be a smooth spacelike Cauchy hypersurface.
Then the Cauchy time function $T$ can be chosen in such a way that $V_0$ is one of its levels.

Using Theorem~\ref{thm:FlexLem1} we will now show that \emph{one can also prescribe the lapse function along $V_0$}.

Let $\check N:V_0\to\R$ be smooth and positive.
Let $\hat T$ be a smooth function defined on a neighborhood $U$ of $V_0$ which coincides with $T$ on $V_0$ and such that $\grad \hat T = \check N^{-1} \cdot \nu$ where $\nu$ is the past-directed timelike unit normal field along $V_0$.
By shrinking $U$ if necessary we can arrange that the gradient of $\hat T$ is past-directed timelike on all of $U$ and that there are smooth spacelike Cauchy hypersurfaces $\Sigma_-$ and $\Sigma_+$ in $V$ which lie in the causal past and future of $U$, respectively.

Now put $f_0=T$ and $\Ft(t,v) = t\hat T(v) + (1-t)T(v)$.
Since the cone of past-directed timelike tangent vectors is convex, the function $F(t,\cdot)$ has a past-directed timelike gradient field on $U$ for each $t\in [0,1]$.

Having a past-directed timelike gradient field imposes an open first order differential relation on functions on $V$.
We apply Theorem~\ref{thm:FlexLem1} with $k=1$ and obtain a smooth function $\ft:[0,1]\times V\to\R$ such that $f(0) = f_0 = T$, each $f(t)$ has past-directed timelike gradient field, coincides with $f_0$ outside $U$ and coincides with $F(t)$ on a smaller neighborhood of $V_0$.
\bigskip

\noindent
\emph{Claim:}
Each $f(t)$ is a Cauchy time function, i.e.\ its nonempty level sets are Cauchy hypersurfaces.

\begin{proof}
Fix $t$ and write $h=f(t)$ for brevity.
Let $c:(0,1)\to V$ be an inextendible future-directed timelike curve.
Since the gradient of $h$ is timelike past-directed and the velocity vector of $c$ is timelike future-directed the function $h$ increases strictly along $c$.
Thus each level of $h$ is hit at most once by $c$.

Moreover, $c$ intersects $\Sigma_-$ and $\Sigma_+$ at points $c(s_-)$ and $c(s_+)$, $0<s_-<s_+<1$, as $\Sigma_\pm$ are Cauchy hypersurfaces.
Thus the level sets of $h$ for values in $[h(c(s_-)),h(c(s_+))]$ intersect $c$ as well.
For the levels below $h(c(s_-))$ and above $h(c(s_+))$ this is also true because $h$ coincides with the Cauchy time function $T$ in the past of $\Sigma_-$ and in the future of $\Sigma_+$.
\end{proof}

Now consider the Cauchy time function $\check T = F(1)$.
Since $\check T$ coincides with $\hat T$ near $V_0$ we have along $V_0$
$$
g(\grad\check T,\grad\check T)
=
g(\grad\hat T,\grad\hat T)
=
g(\check N^{-1}\nu,\check N^{-1}\nu)
=
-\check N^{-2}.
$$
Hence if we replace $T$ by $\check T$ in \eqref{eq:BSMetric} then the lapse function will be $\check N$ along $V_0$.
We have deformed a given Cauchy time function through Cauchy time functions into one which has prescribed lapse function along a given level set.
This procedure can be repeated and yields prescribed lapse functions along finite or countable families of Cauchy hypersurfaces as long as they do not intersect nor accumulate.

\section{Counter-intuitive approximations}

Now we turn to our main application of local flexibility, the construction of sections which have possibly very restrictive local properties on open dense subsets.
Typically, it is impossible to achieve this on all of $V$.

\subsection{The approximation theorem}
In order to formulate it precisely we consider the following setting:

\begin{setting}
\label{setting2}
We denote by
\begin{enumerate}[\myicon]
\item 
$V$ a smooth manifold;
\item
$\pi:X\to V$ a smooth vector bundle;
\item
$k\in\N$ a positive integer;
\item
$\Gamma$ a subsheaf of the sheaf of $C^k$-sections of $X$;
\item
$f$ a $C^{k}$-section on $V$;
\item
$N$ a neighborhood of $f$ in the strong $C^{k-1}$-topology.
\end{enumerate}
\end{setting}

Recall the commutative diagram
$$
\xymatrix{
J^kX \ar[rr]^{\pi_{k,k-1}} \ar[dr]^{\pi_k} && J^{k-1}X \ar[dl]_{\pi_{k-1}} \\
& V &
}
$$
Since $\pi:X\to V$ is a vector bundle, $\pi_k$ and $\pi_{k-1}$ are vector bundle projections as well while $\pi_{k,k-1}:J^kX \to J^{k-1}X$ is an affine bundle.

\begin{notation}
We set $J^k\Gamma := \{j^k\gamma(p) \mid \gamma \mbox{ is a local section of $\Gamma$, defined near $p$,}\, p\in V\}\subset J^kX$.
\end{notation}

\begin{satz}\label{thm:solveondense}
Suppose we are in Setting~\ref{setting2} and assume that for each $p\in V$ there is an open neighborhood $W$ of $j^{k-1}f(p)$  in $J^{k-1}X$ and a map $\sigma_W:W\to J^kX$ such that
\begin{enumerate}[\myicon]
\item 
$\sigma_W$ maps compact subsets of $W$ to relatively compact subsets of $J^k X$ (this holds for example if $\sigma_W$ is continuous); 
\item 
$\pi_{k,k-1}\circ\sigma_W = \id_W$;
\item
$\sigma_W(\omega)\in J^k\Gamma$ for each $\omega\in W$.
\end{enumerate}
Then there exists a section $\hat f$ of $X\to V$ and an open dense subset $\UU\subset V$ with the following properties:
\begin{enumerate}[\myicon]
\item 
$\hat f\in C^{k-1,1}_\loc(V,X)$;
\item
$\hat f\in N$;
\item
$\hat f|_\UU \in \Gamma(\UU)$.
\end{enumerate}
\end{satz}

For the proof of the theorem we need the following lemma:

\begin{lem}\label{lem:Q}
Under the assumptions of Theorem~\ref{thm:solveondense} there exists an open neighborhood  $\RR$ of the image of $j^k f$ in $J^k X$ such that $\RR \cap (\pi_k)^{-1}(K)$ is relatively compact in $J^k X$ for every compact $K\subset V$  and the following ``convexity'' condition holds: 

Let $p \in V$ and let $\phi$ be a $C^k$-germ of sections around $p$ solving $\RR$.
Then there exists an open neighborhood $U$ of $p$ and an $F \in C^0([0,1], C^k(U, X))$ such that 
\begin{enumerate}[\myicon]
\item 
$F(0)$ represents $\phi$; 
\item
$F(t) \in C^k(U, X)$ solves $\RR$ over $U$  for all $t  \in [0,1]$; 
\item 
$F(1) \in \Gamma(U)$;
\item
$j^{k-1}F(t)(p)  = j^{k-1}\phi(p)$ for all $t \in [0,1]$.
\end{enumerate}
\end{lem}

\begin{proof}[Proof of Lemma~\ref{lem:Q}]
We fix an auxiliary Euclidean inner product on the vector bundle $\pi_k:J^kX \to V$ with induced fiber norm $\|\cdot\|_{J^kX}$.
At each $p\in V$ the restriction of $\pi_{k,k-1}$ to the orthogonal complement of $\ker(\pi_{k,k-1})\subset \pi_k^{-1}(p)$ is a linear isomorphism onto $\pi_{k-1}^{-1}(p)$.
Thus
$$
\sigma := \big(\pi_{k,k-1}|_{\ker(\pi_{k,k-1})^\perp}\big)^{-1}:J^{k-1}X \to J^kX
$$
defines a global smooth section of $\pi_{k,k-1}:J^kX \to J^{k-1}X$.

Choose a neighborhood $\WW$ of the image of $j^{k-1}f$ in $J^{k-1}X$ which is a union of open sets $W$ as in the statement of Theorem~\ref{thm:solveondense}.
Shrinking $\WW$ if necessary, we can assume w.l.o.g.\ that it is relatively compact over each compact $K\subset V$.

On such a $W$ the function $J^{k-1}X\to\R\cup\{\infty\}$ defined by $\omega\mapsto \mathrm{dist}(\sigma(\omega),J^k\Gamma\cap \pi_{k,k-1}^{-1}(\omega))$ satisfies
$$
\mathrm{dist}(\sigma(\omega),J^k\Gamma\cap \pi_{k,k-1}^{-1}(\omega)) \le \|\sigma(\omega)-\sigma_W(\omega)\|_{J^kX}
$$
and is hence locally bounded.
Therefore we can find a continuous function $R:\WW\to \R$ such that $\mathrm{dist}(\sigma(\omega),J^k\Gamma\cap \pi_{k,k-1}^{-1}(\omega)) < R(\omega)$ for all $\omega\in\WW$.
By increasing $R$ if necessary, we ensure that 
$$
\|j^kf-\sigma(j^{k-1}f)\|_{J^kX}<R(j^{k-1}f)
$$ 
for the given section $f$.

We define $\RR$ by
$$
\RR := \{\Omega\in \pi_{k,k-1}^{-1}(\WW) \mid \|\Omega-\sigma(\pi_{k,k-1}(\Omega))\|_{J^kX}<R(\pi_{k,k-1}(\Omega)) \} .
$$
Then $\RR$ is an open neighborhood of the image of $j^kf$ in $J^kX$.
For each compact $K\subset V$ the set 
$$
\RR\cap \pi_k^{-1}(K) 
=
\{\Omega\in \pi_{k,k-1}^{-1}(\WW\cap \pi_{k-1}^{-1}(K)) \mid \|\Omega-\sigma(\pi_{k,k-1}(\Omega))\|_{J^kX}<R(\pi_{k,k-1}(\Omega)) \}
$$
is relatively compact.
Let $\varphi$ be a $C^k$-section of $X\to V$, defined on a neighborhood of $p\in V$ such that the image of $j^k\varphi$ is contained in $\RR$.
By construction, $\RR$ intersects $J^k\Gamma\cap \pi_{k,k-1}^{-1}(j^{k-1}\varphi(p))$.
Pick a $k$-jet in $\RR\cap J^k\Gamma\cap \pi_{k,k-1}^{-1}(j^{k-1} \varphi(p))$ and represent it by a local section $\gamma$ of $\Gamma$.
The straight line segment $(1-t)j^k\varphi(p)+tj^k\gamma$ is entirely contained in $\RR\cap \pi_{k,k-1}^{-1}(j^{k-1}\varphi(p))$ by convexity of norm balls.
Thus 
\begin{equation}
F(t):=(1-t)\varphi+t\gamma
\label{eq:deform22}
\end{equation}
has all the required properties if the neighborhood $U$ of $p$ is chosen sufficiently small.
\end{proof}

\begin{proof}[Proof of Theorem~\ref{thm:solveondense}]
We choose an open neighborhood $\NN$ of $\im j^{k-1}f$ in $J^{k-1}X$ such that $\{h\in C^{k-1}(V,X) \mid j^{k-1}h(p)\in\bar{\NN},\, p\in V\} \subset N$.
Pick some $\RR\subset J^k X$ as in Lemma~\ref{lem:Q}.

We provide the jet bundles of $X$ with fiber metrics and induced norms $\|\cdot\|_{J^mX}$ so that the usual $C^m$-norms of sections of $X$ are defined as $\|u\|_{C^m(V)}=\sup_{V}\|j^mu\|_{J^mX}$.

Let $\{p_1,p_2,p_3,\ldots\}$ be a countable dense subset of $V$. 
We construct a sequence $(f_\nu)_{\nu=0,1,\ldots}$ of $C^k$-sections of $X$ together with open neighborhoods $U_\nu$ of $p_\nu$ for $\nu= 1, 2, \ldots$ such that the following holds:
\begin{enumerate}[\myicon]
\item $f_0=f$; 
\item $U_\nu \supset U_{\nu-1}$;
\item $f_{\nu} = f_{\nu-1}$ on $U_{\nu-1}$; 
\item $f_{\nu}$ solves $\RR \cap (\pi_{k,k-1})^{-1}(\NN)$ over $V$;
\item There is a neighborhood $U$ of $\overline{U}_\nu$ such that $f_{\nu}|_U\in\Gamma(U)$;
\item $\|f_\nu-f_{\nu-1}\|_{C^{k-1}(V)}<2^{-\nu}$.
\end{enumerate} 
Assume that $f_{\nu-1}$ has been constructed, together with $U_{\nu-1}$ where $\nu \geq 1$. 
If  $p_{\nu} \in \overline{U}_{\nu-1}$ then, by the inductive assumption, there is an open neighborhood $U_\nu$ of $\overline{U}_{\nu-1}$ such that $f_{\nu-1}|_{U}\in\Gamma(U)$ for a neighborhood $U$ of $\overline{U}_{\nu}$. 
We then simply put $f_\nu := f_{\nu-1}$.

Now assume $p_{\nu} \notin \overline{U}_{\nu-1}$.
We consider a local deformation $F_{p_\nu}\in C^0([0,1],C^k(U_{p_\nu},X))$ as in Lemma~\ref{lem:Q} for the germ represented by $\phi := f_{\nu-1}$ around $p_{\nu}$.
By shrinking $U_{p_\nu}$ if necessary we can assume that $U_{p_\nu}$ is disjoint from $\overline{U}_{\nu-1}$, that $F_{p_\nu}(t)$ solves $\NN$ for $t \in [0,1]$ and that $\|j^{k-1}F_{p_\nu}-j^{k-1}f_{\nu-1}\|_{J^{k-1}X}<2^{-\nu}$. 
For the second and the last requirement we recall  that $\NN$ is an open subset of $J^{k-1}X$ and $j^{k-1} F_{p_\nu}(t)(p)$ is constant in $t$. 

We apply Theorem~\ref{thm:FlexLem1} to the section $f_{\nu-1}$, to $V_0=\{p_\nu\}$, to the deformation $F_{p_\nu}$ and the open PDR 
$$
\RR_\nu := \RR \cap (\pi_{k,k-1})^{-1}(\NN) \cap \{\omega\in J^k X \mid \|\pi_{k,k-1}(\omega)-j^{k-1}f_{\nu-1}\|_{J^{k-1}X} < 2^{-\nu}\}  .
$$
We obtain an open neighborhood $U_{p_\nu,0} \subset U_{p_\nu}$ of $p_\nu$ and a global deformation $f_{p_\nu}\in C^0([0,1],C^k(V,X))$ such that $f_{p_\nu}(1)$ solves $\RR_\nu$ over $V$, it coincides with $F_{p_\nu}(1)$ (and hence is a section of $\Gamma$) over a neighborhood of $\overline{U}_{p_\nu,0}$, and it coincides with $F_{p_\nu}(0)$ (and hence with $f_{\nu-1}$) over a neighborhood of $\overline{U}_{\nu-1}$.
Put $U_\nu := U_{\nu-1}\cup U_{p_\nu,0}$.
Then $f_\nu$ is a section of $\Gamma$ over a neighborhood of $\overline{U}_\nu$.

Moreover, we have $\|f_\nu-f_{\nu-1}\|_{C^{k-1}(V)}<2^{-\nu}$.
This implies that $(f_\nu-f_0)_\nu$ is a Cauchy sequence in the space of $C^{k-1}$-sections with bounded derivatives up to order $k-1$.
Thus there is a limit section in $C^{k-1}(V,X)$ which we denote as $\hat f-f_0$.
By the properties of $\RR$, the derivatives of order $k$ of $f_\nu-f_0$ are locally uniformly bounded.
Hence the derivatives of order $k-1$ are Lipschitz with locally uniform Lipschitz constant.
Such a Lipschitz bound persists under uniform convergence, thus $\hat f-f_0\in C^{k-1,1}_\loc(V,X)$.
Since $f_0\in C^k(V,X)$ we conclude $\hat f\in C^{k-1,1}_\loc(V,X)$.

The set $\UU:=\bigcup_{1 \leq \mu <\infty} U_\mu$ is open and dense in $V$.
Since the sequence $(f_\nu)$ is eventually constant (in $\nu$) on each $U_\mu$, the limit section $\hat f$ satisfies $\hat f|_{\UU}\in\Gamma(\UU)$.

Finally, since each $f_\nu$ solves $\NN$ the image of $j^{k-1}\hat f$ is contained in $\bar\NN$ and hence $\hat f\in N$.
\end{proof}

\subsection{Lipschitz functions}
Theorem~\ref{thm:solveondense} can be used in many different contexts to derive counterintuitive approximation results.
Let us start with a relatively elementary example, namely real-valued functions.

\begin{cor}\label{cor:Lipschitz}
Let $f:[0,1]\to\R$ be a $C^1$-function, let $\eps>0$ and let $K\in\R$.
Then there exists a Lipschitz function $\hat f:[0,1]\to\R$ such that 
\begin{enumerate}[\myicon]
\item 
$|f-\hat f|<\eps$;
\item
$\hat f$ is smooth and satisfies $\hat f'=K$ on an open dense subset of $[0,1]$.
\end{enumerate}
\end{cor}

\begin{proof}
We extend $f$ to a $C^1$-function, again denoted $f$, to $\R$.  
We apply Theorem~\ref{thm:solveondense} with the following choices in Setting~\ref{setting2}:
$V=\R$, $X$ is the trivial line bundle so that sections are nothing but real-valued functions, $k=1$, and $\Gamma$ is the sheaf of smooth functions with constant derivative $K$. 
The strong $C^0$-neighborhood of $f$ is given by $N=\{h\in C^0(\R) \mid |f-h|<\eps\}$.

Theorem~\ref{thm:solveondense} applies because we can put $W:=J^0X=X=V\times\R$ and $\sigma_W(p,\xi)$ defined to be the $1$-jet of the affine function $t\mapsto K\cdot(t - p) + \xi$.
The function $\hat f:\R\to\R$ given by Theorem~\ref{thm:solveondense} is locally Lipschitz, hence its restriction to $[0,1]$ is Lipschitz.
\end{proof}

If we apply this corollary to $f(t)=t$, $K=0$ and $\eps=0.0001$ then we get a Lipschitz function $\hat f:[0,1]\to\R$ with $\hat f(0)<0.0001$, $\hat f(1)>0.9999$ and $\hat f'=0$ on an open dense subset.
Note that Lipschitz functions are differentiable almost everywhere by Rademacher's theorem and the fundamental theorem of calculus holds.
Thus we have 
$$
\int_0^1 \hat f'(x)\, dx = \hat f(1)-\hat f(0) > 0.9998
$$
which, at first glance, seems to violate $\hat f'=0$ on the open dense subset.
The point is that open dense subsets need not have full measure, so there is no contradiction.
Clearly, $\hat f$ cannot be $C^1$ in this case.

This function is not to be confused with the Cantor function (see e.g.\ \cite{Cantor}), also known as the devil's staircase.
The Cantor function is a H\"older continuous function $[0,1]\to[0,1]$ with H\"older exponent $\alpha=\ln2/\ln3$.
It has vanishing derivative on an open subset of full measure but it is not absolutely continuous.
Hence the fundamental theorem of calculus cannot be applied and the Cantor function is not Lipschitz.

\subsection{Embeddings of surfaces}
Next we approximate embedded surfaces by those with constant Gauss curvature on open dense subsets.

\begin{cor}\label{cor:einbett}
Let $K\in\R$.
Let $V$ be an analytic surface, let $f:V\hookrightarrow \R^3$ be a $C^2$-embedding and let $N$ be a neighborhood of $f$ in the strong $C^1$-topology.

Then there exists a $C^{1,1}$-embedding $\hat f:V\hookrightarrow \R^3$ in $N$ which is analytic on an open dense subset $\UU\subset V$ and has constant Gauss curvature $K$ on $\UU$ (w.r.t.\ the induced metric).
\end{cor}

\begin{proof}
We apply Theorem~\ref{thm:solveondense} with the following choices in Setting~\ref{setting2}:
Let $X$ be the trivial $\R^3$-bundle so that sections are maps $V\to \R^3$.
Let $k=2$ and $\Gamma$ be the sheaf of analytic maps $w$ satisfying
$$
\det\bigg(\Big\langle\frac{\partial^2w}{\partial u^i\partial u^j},\frac{\partial w}{\partial u^1}\times \frac{\partial w}{\partial u^2}\Big\rangle\bigg)
=
K\cdot \det\bigg(\Big\langle \frac{\partial w}{\partial u^i},\frac{\partial w}{\partial u^j}\Big\rangle\bigg)\cdot \Big|\frac{\partial w}{\partial u^1}\times \frac{\partial w}{\partial u^2}\Big|
$$
in local coordinates $(u^1,u^2)$ on $V$.
If $\frac{\partial w}{\partial u^1}$ and $\frac{\partial w}{\partial u^2}$ are linearly independent then this condition is equivalent to the induced Gauss curvature being $K$ and otherwise it is void.
Since the set of embeddings is open in the strong $C^1$-topology (\cite[Thm.~1.4]{Hirsch94}) we can assume w.l.o.g.\ that all maps in $N$ are embeddings, by shrinking $N$ if necessary.

To see that Theorem~\ref{thm:solveondense} applies let $D\subset\R^2$ be an open disk about the origin and let $h:D\to\R$ be an analytic function such that $h(0)=0$, $\nabla h(0)=0$, and the graph of $h$ is a surface of constant Gauss curvature $K$ in $D\times\R$.
Let $(U,x^1,x^2)$ be a local analytic chart of $V$.
We put 
$$
W:=\{\omega\in\pi_1^{-1}(U) \mid \mbox{the differential of (a map representing) $\omega$ is injective}\}.
$$
Now given $\omega\in W$ put $p:=\pi_1(\omega)$ and represent $\omega$ by an analytic map $\phi$ defined near $p$.
By shrinking the domain we can ensure that $\phi$ is an embedding. 
Let $A_\omega$ be the unique special-orthogonal matrix $A_\omega\in\mathsf{SO}(3)$ with 
\begin{align*}
A_\omega e_1 
&= 
\lambda\cdot \frac{\partial \phi}{\partial x^1}(x(p)), \quad \lambda>0,\\
A_\omega e_2 
&= 
\mu\cdot \frac{\partial \phi}{\partial x^1}(x(p)) + \nu\cdot \frac{\partial \phi}{\partial x^2}(x(p)), \quad \nu>0.
\end{align*}
Here $e_1,e_2,e_3$ denote the standard basis of $\R^3$.
The matrix $A_\omega$ is uniquely determined by $\omega$ (and the coordinate system) and depends continuously on $\omega$.
Consider the Euclidean motion $E_\omega:\R^3\to\R^3$ defined by $E_\omega x=A_\omega x + \phi(p)=A_\omega x + \pi_{1,0}(\omega)$.
The map
$$
S_\omega
:=
E_\omega \circ 
\begin{pmatrix}(E_\omega^{-1}\circ \phi)_1\\
               (E_\omega^{-1}\circ \phi)_2\\
               h\big((E_\omega^{-1}\circ \phi)_1,(E_\omega^{-1}\circ \phi)_2\big)
\end{pmatrix}
$$
is analytic, parametrizes a surface with constant Gauss curvature $K$, and has $1$-jet $\omega$ at $p$.
Thus 
$$
\sigma_W(\omega) := j^2(S_\omega)(p) \in J^k\Gamma
$$
is a local section as required in the assumptions of Theorem~\ref{thm:solveondense}.
\end{proof}

Corollary~\ref{cor:einbett} is an extrinsic companion to Corollary~\ref{satz:approxpos} below.
It does not contradict the Gauss-Bonnet theorem, see Remark~\ref{rem:GaussBonnet}.

\section{Deforming Riemannian metrics}

In this final section we apply our results to Riemannian metrics.

\begin{cor}\label{satz:approxpos}
Let $V$ be a differentiable manifold of dimension $n \ge 2$ and let $g$ be a $C^2$-Riemannian metric on $V$.
Let $N$ be a neighborhood of $g$ in the strong $C^1$-topology.
Let $K\in\R$.

Then there exists a Riemannian metric $\hat g$ on $V$ with the following properties:
\begin{enumerate}[\myicon]
\item
$\hat g$ has local $C^{1,1}$-Lipschitz regularity;
\item 
$\hat g \in N$;
\item
$\hat g$ is smooth and has constant sectional curvature equal to $K$ on an open dense subset of $V$.
\end{enumerate}
\end{cor}

\begin{proof}
We apply Theorem~\ref{thm:solveondense} with the following choices in Setting~\ref{setting2}:
Let $\pi : X \to V$ be the vector bundle of symmetric $(2,0)$-tensors and let $k = 2$.
Let $\Gamma$ be the sheaf of smooth Riemannian metrics of constant sectional curvature $K$. 

The see that Theorem~\ref{thm:solveondense} applies let $g^{[K]}_{\R^n}$ be the metric of constant sectional curvature $K$ on an open ball about the origin, expressed in normal coordinates.
Then $g^{[K]}_{\R^n}$ has the same $1$-jet at $0$ as the Euclidean metric and orthogonal transformations are isometries.
For any $n$-dimensional Euclidean vector space $Y$ we can choose a linear isometry $A:Y\to\R^n$ and pull the metric back, $g^{[K]}_{Y} := A^*g^{[K]}_{\R^n}$.
The metric $g^{[K]}_{Y}$ does not depend on the choice of $A$.

Next pick a local chart $(U,x^1,\ldots,x^n)$ on $V$.
Put 
$$
W:=\{\omega\in\pi_{1}^{-1}(U) \mid \pi_{1,0}(\omega)\mbox{ is positive definite}\}.
$$
The local section $\sigma_W$ is defined as follows:
Express the $1$-jet $\omega\in \pi_1^{-1}(U)$ in the given coordinates as $\omega = \omega_0 + \sum_j \omega_j x^j$ and associate the metric $h_\omega$ given by this formula, $h_\omega = \omega_0 + \sum_j \omega_j x^j$, defined in a neighborhood of $p:=\pi_1(\omega)$.
Clearly, $j^1(h_\omega)(p)=\omega$.
Denote the exponential map of $h_\omega$ at $p$ by $\exp_p^{h_\omega}$.
Now put
$$
\sigma_W(\omega) := j^2\Big(\big((\exp_p^{h_\omega})^{-1}\big)^*g^{[K]}_{T_pV}\Big)(p) \in J^k\Gamma.
$$
Observe that indeed
$$
\pi_{2,1}(\sigma_W(\omega))
=
j^1\Big(\big((\exp_p^{h_\omega})^{-1}\big)^*g^{[K]}_{T_pV}\Big)(p)
=
j^1(h_\omega)(p)
=
\omega
$$
because $j^1((\exp_p^{h_\omega})^*h_\omega)(0) = j^1(g_\mathrm{eucl})(0)=j^1(g^{[K]}_{T_pV})(0)$.
\end{proof}

\begin{rem}
Corollary~\ref{satz:approxpos} would be false if we demanded that $\hat g$ had regularity $C^2$ on all of $V$.
Then the sectional curvature of $(V,\hat g)$ would be continuous and $\sec_{\hat g}\equiv K$ on a dense subset would imply that this holds on all of $V$.
But most $V$ do not admit such a metric.
\end{rem}

\begin{rem}\label{rem:AlexandrovBonnetMyers}
Even if the metric $\hat g$ on $V$ has constant sectional curvature $1$ on an open dense subset $\UU$, it cannot, in general, have curvature $\ge 1$ in the sense of Alexandrov spaces on all of $V$.
Namely, this implies that the diameter of $(V,\hat g)$ is bounded above by $\pi$, at least if $V$ is compact, see e.g.\ \cite[Thm.~10.4.1]{BBI}.
Now if $\diam(V,g)$ is much larger than $\pi$ then this contradicts $\hat g$ being $C^1$-close (and hence $C^0$-close) to~$g$.

In fact, if a metric has curvature $\ge 1$ in the Alexandrov sense on an open dense subset only, then there is no upper bound on the diameter.
This is illustrated by the following picture:
\begin{center}
\includegraphics[scale=0.2]{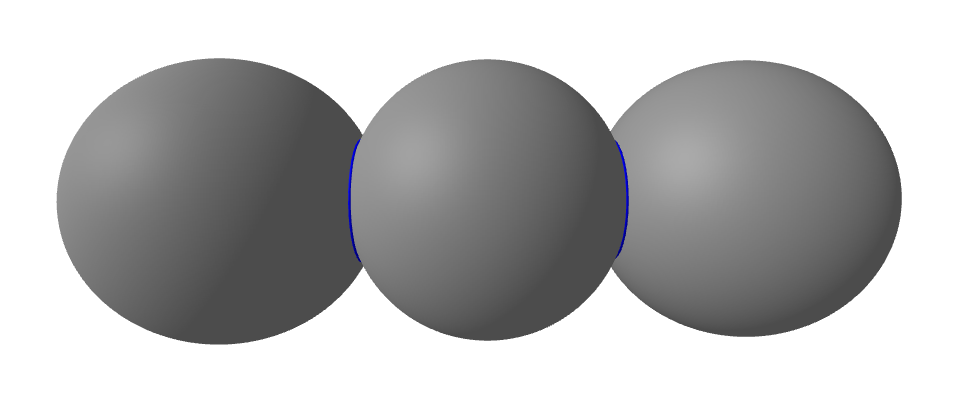}

\Abb{Space with curvature $\ge 1$ on dense open subset but $\diam>\pi$}
\end{center}
\end{rem}

\begin{rem}\label{rem:GaussBonnet}
Let $V$ be a compact surface of higher genus.
By Corollary~\ref{satz:approxpos} we can find a $C^{1,1}$-metric on $V$ whose Gauss curvature (which is defined as an $L^\infty$-function on $V$) satisfies $K\equiv 1$ on an open dense subset $\UU\subset V$.
The Gauss-Bonnet theorem (which holds for $C^{1,1}$-metrics, see Appendix~\ref{sec-AnhangC}) says $\int_V K \, d A = 2\pi\chi(V) < 0$.
This may seem to be a contradiction but, again, open dense subsets need not have full measure.
\end{rem}

\begin{rem}
Similarly, we can refine Remark~\ref{rem:AlexandrovBonnetMyers}.
Let $V$ be compact for simplicity.
If $\UU\subset V$ has full measure then the Bonnet-Myers theorem for $C^{1,1}$-metrics applies (\cite[Thm.~4.1]{Graf}) and we get $\diam(V)\le\pi$.
Hence, if $\diam(V)>\pi$, which we can arrange by Remark~\ref{rem:vorschreiben} below, $\UU$ cannot have full measure.
\end{rem}

\begin{rem}\label{rem:approxneg}
Even if $\hat g$ has constant sectional curvature $-1$ on an open dense subset we cannot demand that the curvature is $\le -1$ the Alexandrov sense on all of $V$.
In this case we would violate Preissmann's theorem \cite[Thm.~9.3.3]{BBI}, for instance.
\end{rem}

\begin{cor}\label{cor:strangemetrics}
Let $K\in\R$.
Each differentiable manifold of dimension $\ge 2$ has a complete $C^{1,1}_\loc$-Riemannian metric which is smooth and has constant sectional curvature $\equiv K$ on an open dense subset.
\end{cor}

\begin{proof}
Choose a complete smooth Riemannian metric $g$ on $V$.
We can choose a neighborhood $N$ of $g$ in the strong $C^1$-topology such that each metric in $N$ is complete.
Now apply Corollary~\ref{satz:approxpos}.
\end{proof}

\begin{rem}\label{rem:vorschreiben}
As discussed above, this is false for $C^2$-metrics.
Since the $C^1$-metric can be chosen $C^1$-close to an arbitrary smooth metric on $V$, we can in Corollary~\ref{cor:strangemetrics} in addition prescribe geometric quantities like volume, diameter, injectivity radius up to arbitrarily small error.
\end{rem}

\begin{rem}
One might be tempted to think that a metric as in Corollary~\ref{cor:strangemetrics} can be constructed as follows:
triangulate the manifold, then equip the open $n$-simplices with metrics of constant sectional curvature $K$ and glue the metrics along the $(n-1)$-skeleton.
Such a procedure will not give a metric of $C^{1,1}_\loc$-regularity.
Indeed, such a metric would satisfy $\sec\equiv K$ on an open dense subset of full measure and hence violate the Gauss-Bonnet theorem, cf.\ Remark~\ref{rem:GaussBonnet}.
\end{rem}

\begin{rem}
Corollary~\ref{cor:strangemetrics} should be contrasted with the implications of Gromov's $h$-principle for diffeomorphism-invariant partial differential relations.
The latter implies that every connected \emph{noncompact} manifold has a \emph{smooth} but \emph{incomplete} Riemannian metric with positive sectional curvature and another one with negative sectional curvature, see \cite[Thm.~4.5.1]{Gromov69}.
\end{rem}

\begin{rem}
In his famous precompactness theorem \cite[Sec.~8.20]{Gromov99} Gromov proves $C^{0,1}$-regularity of the limit metric occurring in that theorem.
This regularity result was later improved to $C^{1,\alpha}$ for all $\alpha\in(0,1)$, see \cite{GW88,Nikolaev,Peters}.
The $C^{1,1}$-regularity shown in Corollary~\ref{satz:approxpos} is the borderline case of this.
It would be interesting to know if this is a coincidence or if there is a deeper relationship.
\end{rem}

\begin{epi} 
Let $(V,g)$ be a Riemannian manifold with sectional curvature $\sec_g > 0$ and let $V_0=\{p\}$ consist of just one point. 
Let $K > 0$. 
In the following we use the notation introduced in the proof of Corollary~\ref{satz:approxpos}. 
Working with the local section $\gamma=  \big((\exp_p^g)^{-1}\big)^*g^{[K]}_{T_p V}$ of $X$ around $p$ the map defined in \eqref{eq:deform22} for $\varphi = g$ locally deforms $g$ through positive curvature metrics into one of constant positive curvature $K$, keeping the $1$-jet constant at $p$. 
Theorem~\ref{thm:FlexLem1} for $k = 2$ implies:

\smallskip
\emph{One can deform a Riemannian metric of positive sectional curvature on $V$ through such metrics into one which has constant sectional curvature $K > 0 $ near $p$}.
\smallskip

A similar argument works if $\{p\}$ is replaced by an embedded geodesic $V_0 \subset V$, working with local Fermi coordinates around $V_0$. 
This is an application of Theorem~\ref{thm:FlexLem1} with possibly noncompact $V_0$.
Moreover, we can treat other curvature quantities and curvature bounds.

This discussion extends to families of metrics as follows. 
Let  ${\rm Sec}_+(V)$ be the space of Riemannian metrics of positive sectional curvature on $V$, equipped with the weak $C^{\infty}$-topology, and  let ${\rm Sec}_K(V,p)$  denote the subspace of metrics of constant sectional curvature $K$ in some neighborhood of $p$, equipped with the quasi-$C^{\infty}$-topology induced by the directed system $C^{\infty}(U,X)$, $\{p\} \subset U \subset V$ open, cf.~the remarks preceding Corollary~\ref{cor:homotop}.  
Let
\[
  f_0 :    D^k  \to {\rm Sec}_+(V) 
\]
be a continuous map such that, by definition of quasi-topologies, there is a uniform neighborhood of $p$ on which $f_0 ( \xi)$ has constant sectional curvature $K$ for all $\xi \in \partial D^k$. 

Using compactness of $D^k$  we find an open neighborhood of $0 \subset T_p M$ such that $\exp^{f_0(\xi)}_p : T_p M \to M$ maps this neighborhood diffeomorphically onto an open neighborhood of $p$ for all metrics $f_0(\xi)$, $\xi \in D^k$. 
Let $\{p\} \subset U \subset V$ be an open neighborhood of $p$ which is contained in these neighborhoods for all $\xi$. 
We can choose $U$ so small that working with the $\xi$-dependent sections $\phi(\xi) = f_0(\xi)$ and $\gamma(\xi) = \big((\exp_p^{f_0(\xi)})^{-1}\big)^*g^{[K]}_{T_p V}$ of $X$ over $U$,  Equation~\eqref{eq:deform22} defines a continuous map $F: D^k \to C^{\infty}([0,1], {\rm Sec}_+(U))$ with $F(\xi)(0)  = f_0(\xi)|_U$,  and  $F(\xi)(1) = \gamma(\xi)$ for all $\xi \in D^k$.
In addition we can assume that $F(\xi)(t)$ is constant in $t$ for $\xi \in \partial D^k$.

By Addendum~\ref{add:family} we find an open neighborhood $\{p\} \subset U_0 \subset U$ and a continuous map $f: D^k \to C^{\infty}([0,1] , {\rm Sec}_+(V))$ such that for all $\xi \in D^k$ the deformations $f(\xi)$ and $F(\xi)$ coincide on $U_0$ and for all $\xi \in \partial D^k$ the deformation $f(\xi)$ is constant.  
This shows: The inclusion 
\[
    {\rm Sec}_K (V,p) \hookrightarrow {\rm Sec}_+(V)
\]
is a weak homotopy equivalence. 
\end{epi} 

\appendix
\section{\texorpdfstring{Proof of local flexibility for $k=0$}{Proof of local flexibility for k=0}}
\label{sec-AnhangA}

Let  $f_0\in C^\kappa(V,X)$, and $F\in C^\ell([0,1],C^\kappa(U,X))$, where $\ell, \kappa\in\{0,1,\ldots,\infty\}$.
We choose a complete Riemannian metric on $V$ in such a way that $\{v\in V \mid r(v)\le 2\}\subset U$ where $r$ is the distance function from $V_0$ w.r.t.\ this metric.
Let $\tau: \R \to \R$ be a $C^{\infty}$-function with $\tau(r)=1$ for $r\le 1$, $\tau(r)  =0$ for $r\ge 2 $, and $0\le \tau  \le 1$ everywhere.
Furthermore we find a smooth function $r^*  \in C^{\infty}(U \setminus V_0)$ with $r \leq r^* \leq r + 1/2$. 

Since  $\RR \subset X$ is open we can replace  $F$ by a map $F^* \in C^{\ell}([0,1], C^{\kappa}(U,X))$ with the following properties: 
\begin{enumerate}[\myicon]
\item 
For all $t \in [0,1]$ we have $F^*(t) = F(t)$ on $\{ r^*  < 1 \} \cup \{ r^*  > 2 \}$; 
\item 
$F^*(0) = F(0)$; 
\item
$F^*$ induces a map in $C^{\infty}([0,1], C^{\kappa}( \{ 1  \leq r^* \leq 2  \} ,X))$. 
\item
$F^*(t)$ solves $\RR$ over $U$ for all $t \in [0,1]$. 
\end{enumerate}

Now let  $f \in C^0([0,1], C^0(V,X))$ be defined by 
\begin{equation*} 
f(t)(v) := 
\begin{cases}
F^* \big(t\cdot \tau(r^*(v)\big)(v) & \mbox{ if $r(v)<2$,}\\
f_0(v) & \mbox{ if $r(v)\ge2$}.
\end{cases}
\end{equation*}
This is possible since $r(v) \geq 2$ implies $r^*(v) \geq r(v) \geq 2$ and hence   $\tau(r^*(v)) = 0$.  
We have $f \in C^{\ell}([0,1], C^{\kappa}(V,X))$  by the choice of $F^*$. 
Also, for $r(v) < 1/2$ we have $r^*(v) < 1$, which implies $\tau(r^*(v)) = 1$ and hence $f(t)(v) = F^*(t)(v) = F(t)(v)$ for all $t \in [0,1]$.

This means that $f$ has all the required properties  with  $U_0=\{r<1/2\}$ and Theorem~\ref{thm:FlexLem1} is proved for $k=0$.

\begin{rem}
If $\ell\ge \kappa$ then one need not introduce $F^*$ and can simply put 
$$
f(t)(v) = F \big(t\cdot \tau(r^*(v)\big)(v)
$$ 
for $r(v)<2$.
\end{rem}

\section{Proof of Lemma~\ref{lem:DefFunction} }
\label{sec-AnhangB}

The function $\hat\tau_{\delta, \eps}$ defined by 

\begin{minipage}{0.48\textwidth}
$$
\hat\tau_{\delta, \eps}(r) =
\begin{cases}
1 & \mbox{ for }r\le\delta \eps,\\
\frac{\ln (r/\eps)}{\ln\delta} & \mbox{ for }\delta \eps \le r \le \eps,\\
0 & \mbox{ for }r\ge \eps,
\end{cases}
$$
\end{minipage}
\hfill
\begin{minipage}{0.48\textwidth}
\begin{center}
\includegraphics[scale=0.3]{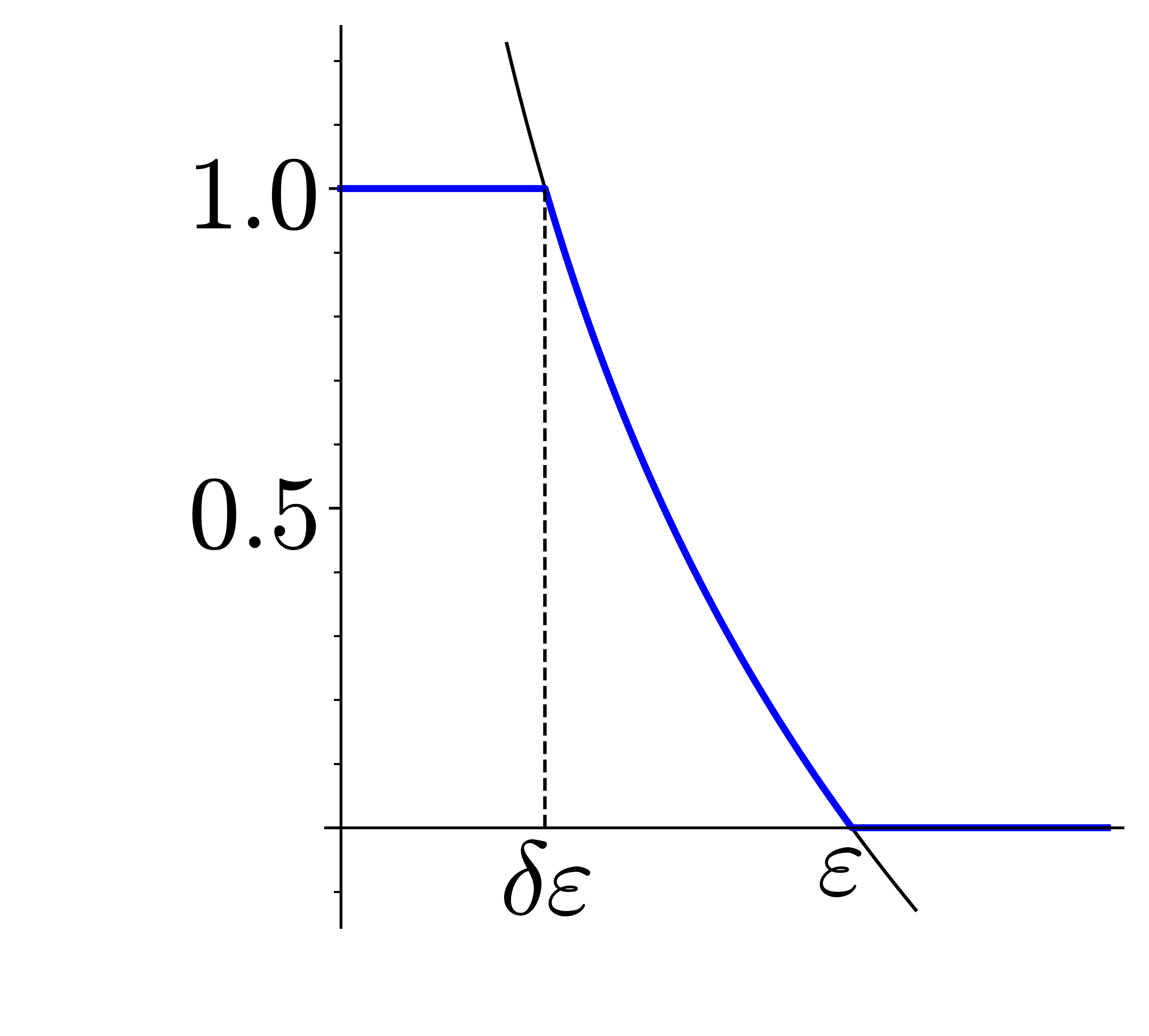}
\end{center}

\vspace{-3mm}
\Abb{The function $\hat\tau_{\delta,\eps}$}
\vspace{2mm}
\end{minipage}
has all properties listed in Lemma~\ref{lem:DefFunction} (with $C_k=1$) except that \eqref{eq:tauepsiv} does not make sense at the points $r=\delta \eps$ and $r=\eps$ where the function is not differentiable.
We need to see that we can smooth $\hat\tau_{\delta, \eps}$ near these two points without destroying properties \eqref{eq:tauepsi}--\eqref{eq:tauepsiv}.

At $r=\eps$ this is easy:
Choose a smooth function $\widehat\ln:(0,\infty)\to\R$ such that 

\begin{minipage}{0.48\textwidth}
\begin{align*}
\widehat\ln(r)&=\ln(r) \mbox{ for $r\le0.9$,}\\
\widehat\ln(r)&=0 \mbox{ for $r\ge1$ and}\\
\widehat\ln&\le 0 \mbox{ everywhere.}
\end{align*}
\end{minipage}
\hfill
\begin{minipage}{0.48\textwidth}
\begin{center}
\includegraphics[scale=0.3]{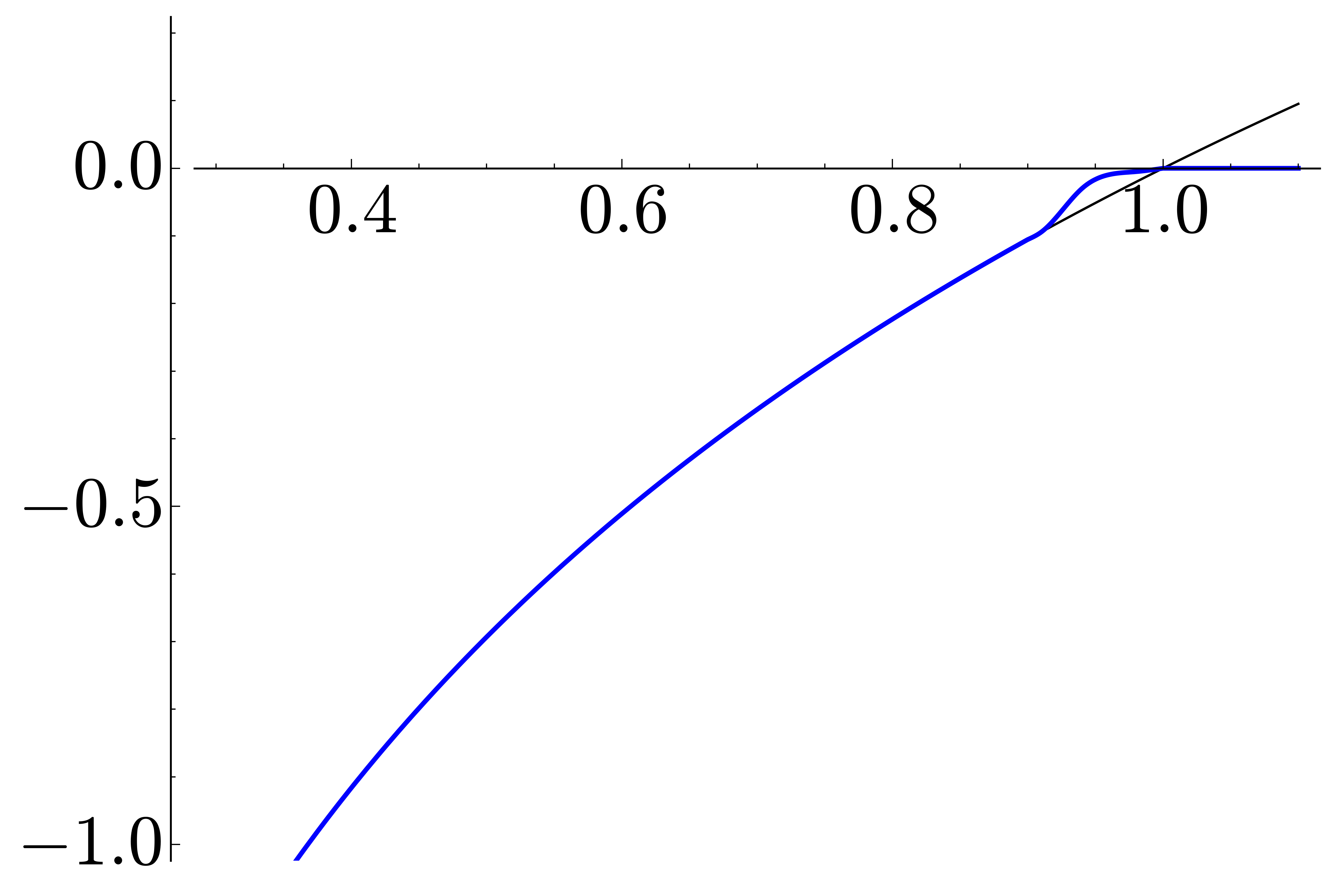}
\end{center}

\vspace{-1mm}
\Abb{The function $\widehat\ln$}
\vspace{2mm}
\end{minipage}
Then we can put $\tau_{\delta, \eps}(r) = \frac{\widehat\ln(r/\eps)}{\ln\delta}$ for $r\ge\frac\eps 2$.
Indeed,
\begin{align*}
\bigg|\frac{d^k}{dr^k}\bigg(\frac{\widehat\ln(r/\eps)}{\ln\delta}\bigg)\bigg|
&=
\bigg|\frac{\widehat\ln^{(k)}(r/\eps)}{\eps^k \ln\delta} \bigg| 
\le
\frac{\big\|\widehat\ln^{(k)}\big\|_{L^\infty[1/2,1]}}{ |\ln\delta| \eps^k} 
\le
\frac{\big\|\widehat\ln^{(k)}\big\|_{L^\infty[1/2,1]}}{|\ln \delta| r^k}
\end{align*}
because $\widehat\ln(r/\eps)$ vanishes for $r>\eps$.

At $r=\delta \eps$ we proceed as follows:
We choose a smooth function $\chi:\R\to\R$ such that

\begin{minipage}{0.48\textwidth}
\begin{align*}
\chi(r)&=1 \mbox{ for $r\le 1$,}\\
\chi(r)&=r \mbox{ for $r\ge 1.1$,}\\
\chi&\le2 \mbox{ on $[1,1.1]$, and}\\
\chi&\ge 1 \mbox{ everywhere.}
\end{align*}
\end{minipage}
\hfill
\begin{minipage}{0.48\textwidth}
\begin{center}
\includegraphics[scale=0.3]{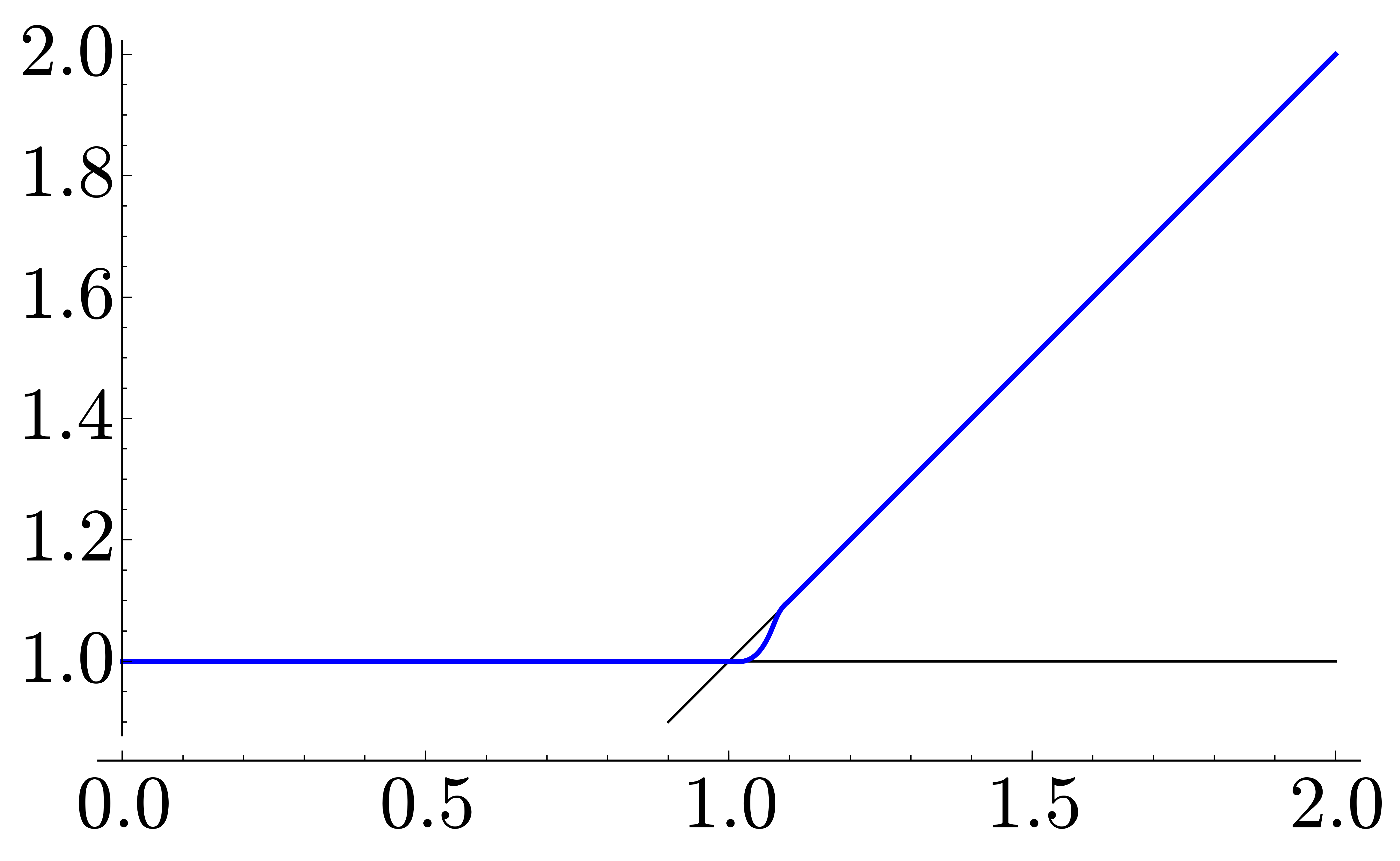}
\end{center}

\vspace{-1mm}
\Abb{The function $\chi$}
\vspace{2mm}
\end{minipage}
Now we put for $r\le\frac{\eps}{2}$:

\begin{minipage}{0.48\textwidth}
$$
\tau_{\delta, \eps}(r) = \frac{\ln(\delta \chi(r/\delta \eps))}{\ln\delta} = 1 + \frac{\ln\chi(r/\delta \eps)}{\ln\delta} .
$$
\end{minipage}
\begin{minipage}{0.48\textwidth}
\begin{center}
\includegraphics[scale=0.35]{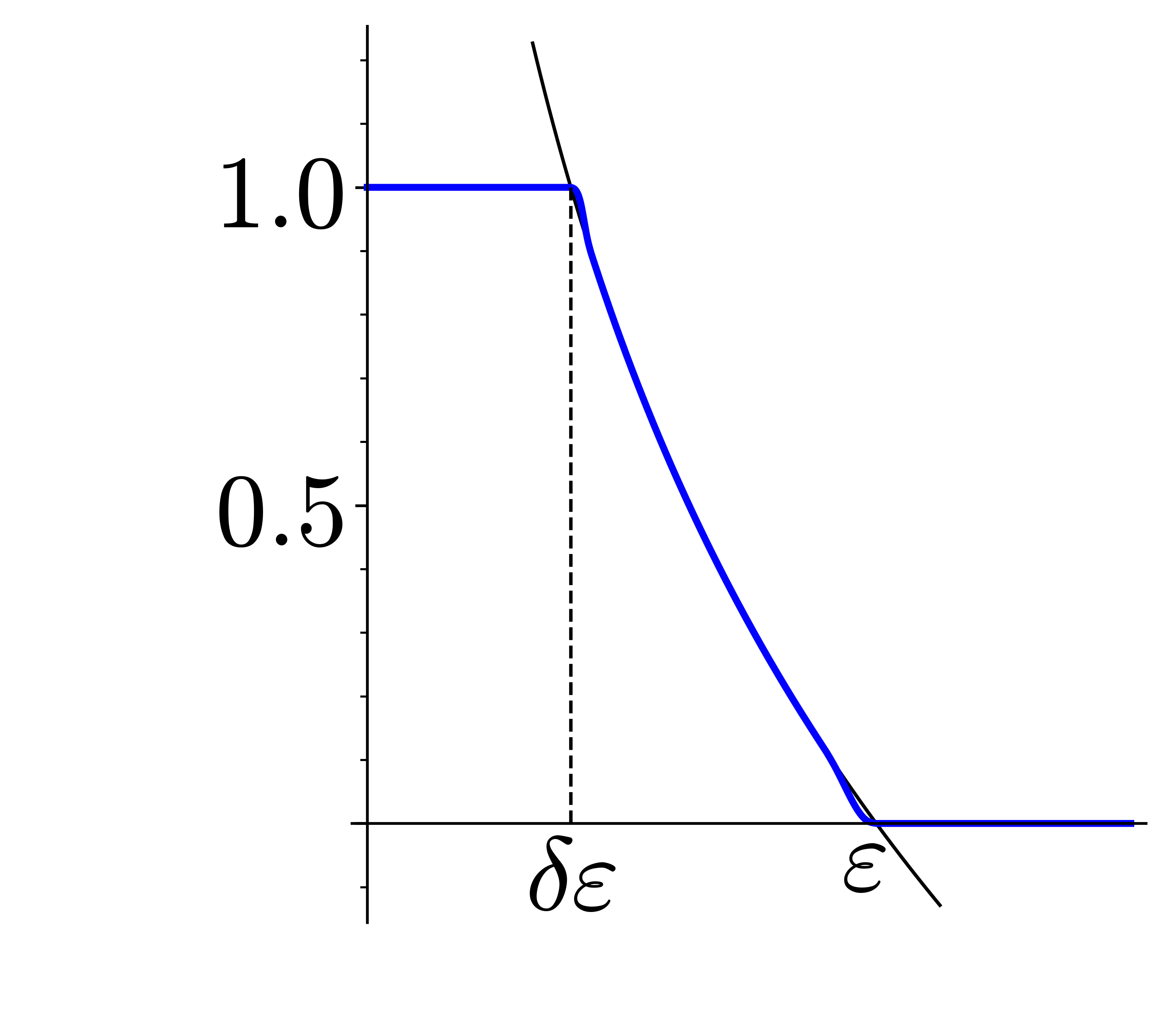}
\end{center}

\vspace{-3mm}
\Abb{The function $\tau_{\delta,\eps}$}
\vspace{2mm}
\end{minipage}
For $r\ge1.1\cdot\delta \eps$ we then have $\tau_{\delta, \eps}(r)=\hat\tau_{\delta, \eps}(r)$ and for $r\le \delta \eps$ we have $\tau_{\delta, \eps}(r)=1$.
It remains to check \eqref{eq:tauepsiv} on $[\delta \eps,1.1\cdot\delta \eps]$.
We compute
$$
\tau_{\delta, \eps}^{(k)}(r) = \frac{1}{\ln\delta}\cdot\frac{d^k}{dr^k}(\ln\circ\chi)(r/\delta \eps)\cdot\frac{1}{(\delta \eps)^{k}}.
$$
Putting $C_k'=\|\frac{d^k}{dr^k}(\ln\circ\chi)\|_{L^\infty[1,1.1]}$ we get for $r\in[\delta \eps,1.1\cdot\delta \eps]$:
\begin{equation*}
|\tau_\eps^{(k)}(r)| 
\le 
\frac{1}{|\ln\delta|}\cdot C_k' \cdot \frac{1}{(\delta \eps)^{k}}
\le
\frac{C_k'\cdot1.1^k}{|\ln\delta|\, r^k}.
\end{equation*}
This concludes the proof.
\hfill$\Box$

\section{Gauss-Bonnet for metrics with low regularity}
\label{sec-AnhangC}

The Gauss-Bonnet theorem for compact surfaces holds for Riemannian metrics with regularity lower than $C^2$.
This is a folklore fact but since it seems hard to find a reference in the literature we provide a proof.

Let $V$ be a compact surface equipped with a Riemannian metric $g$ of Sobolev regularity $H^{2,p}$ with $p>2$.
We choose a sequence $(g_\mu)$ of smooth metrics converging to $g$ in $H^{2,p}$.
Since $p>2$, the sequence also converges in $C^1$ by the Sobolev embedding theorem.

Denoting the Gauss curvature of $g$ by $K_g$ and the area element by $dA_g$, the expressions in local coordinates
\begin{gather*}
K_g = \frac12 \sum_{i,j,k=1}^2 g^{ik}\bigg(\frac{\partial \Gamma_{ik}^j}{\partial x^j} - \frac{\partial \Gamma_{jk}^j}{\partial x^i}\bigg) + \mbox{ lower order terms}    , \\
\Gamma^k_{ij} = \frac12 \sum_{m=1}^2 g^{km} \bigg(\frac{\partial g_{im}}{\partial x^j}+\frac{\partial g_{mj}}{\partial x^i}-\frac{\partial g_{ij}}{\partial x^m}\bigg)  , \\
dA_g = \sqrt{g_{11}g_{22}-g_{12}g_{21}}\, dx^1\, dx^2.
\end{gather*}
show that the second derivatives of $g$ enter linearly in the Gauss-Bonnet integrand $K_g\, dA_g$.
The terms involving no or first-order derivatives converge uniformly and the second derivatives converge in $L^p$ as $\mu\to\infty$.
Thus $K_{g_\mu}\, dA_{g_\mu} \to K_g\, dA_g$ in $L^p$.
In particular, the Gauss-Bonnet integrand of $g$ exists as an $L^p$-density.

Since integration is a bounded linear functional on $L^p$ by the H\"older inequality, we find
$$
\int_V K_{g_\mu}\, dA_{g_\mu} \to \int_V K_g\, dA_g \quad\mbox{ as }\mu\to\infty.
$$
The Gauss-Bonnet theorem for smooth metrics now implies Gauss-Bonnet for $g$.
We have shown:
\smallskip

\emph{The Gauss-Bonnet theorem for compact surfaces holds for Riemannian metrics of Sobolev regularity $H^{2,p}$ with $p>2$ and, in particular, for $C^{1,1}$-metrics.}


\end{document}